\documentclass[11pt]{amsart}
\usepackage[T1]{fontenc}
\usepackage[latin1]{inputenc}
\usepackage{a4wide}
\usepackage{setspace}
\usepackage{xypic}
\usepackage{amssymb}
\usepackage{amsthm}
\usepackage[english]{babel}
\usepackage{latexsym}
\date{}

\usepackage{srcltx}

\newtheorem{thm}{Theorem}[section]

\newtheorem{conjL'}[thm]{Conjecture $\textbf{L}(\overline{\mathcal{X}}_{et},d)_{\geq0}$}

\newtheorem{prop}[thm]{Proposition}

\newtheorem{lem}[thm]{Lemma}
\newtheorem{cor}[thm]{Corollary}

\newenvironment{f-proof}[1][\sc Proof.]{\begin{trivlist}
\item[\hskip \labelsep {\bfseries #1}]}{\hfill{$\square$}\end{trivlist}}

\begin{document}
\title{Orthogonal representations of affine group schemes and twists of symmetric bundles}

\author{Ph. Cassou-Nogu\`es, T. Chinburg, B. Morin, M.J. Taylor  }

\maketitle

\tableofcontents

\section{ Introduction}
Let $K$ be a field of characteristic different from $2$ and let $(M, q)$ be a pair consisting of a finite dimension $K$-vector space $M$,
endowed with a non-degenerate quadratic form $q$. Let $L/K$ be a finite Galois extension with $G=\mbox{Gal}(L/K)$. Any orthogonal
representation, defined  by a group homomorphism $\rho: G\rightarrow {\bf O}(q)$,   gives
rise to a $1$-cocycle in the Galois cohomology set $\check H^{1}(K, {\bf O}(q))$. Since this set classifies the isometry classes of quadratic forms
with the same rank as $q$, one can attach to $(\rho, L)$ a new quadratic form which we shall call the {\it twist of $q$ by $(\rho, L)$}.
Following  a result of Serre on the trace form of  finite separable extensions  in [Se1],  Fr\"ohlich initiated  a series of papers dedicated
to the study of such twists. In [F] he showed  that the twist of the form is not just provided by the abstract definition described above, but can also be described by a simple explicit formula. Moreover, he obtained some beautiful comparison formulas between the Hasse-Witt invariants of the form and its twists,
involving the  Stiefel-Whitney class  and the spinor norm class  of $\rho $.

The work of Fr\"ohlich and Serre has been generalized in various ways, either by providing comparison formulas between the
invariants attached to different kind of representations or interesting quadratic forms as in  [Sa1] and [Sa2] or by considering the case of
higher dimensions ([CNET1], [CNET2], [EKV], and  [K]). One can develop a theory of symmetric bundles over a scheme of arbitrary
dimension [Kne], which are equivariant for the action of a group scheme.   A notion of the  twist of a form can be
defined  in this framework; we now describe this construction. We fix a scheme $Y$ locally of finite type,  in which $2$ is invertible and a finite and flat
group scheme $G$ over $Y$. We consider a symmetric bundle $(M, q)$ over $Y$ which is the underlying bundle  of an orthogonal representation
of $G$,  given by a morphism of group schemes  $\rho: G\rightarrow {\bf O}(q)$. Following Milne, [M], III, Section 4, one can associate to any
group scheme $H$ over $Y$ a cohomological set $\check H^{1}(Y, H)$.  The set $\check H^{1}(Y, {\bf O}(q))$ is seen to classify the isometry classes
of symmetric bundles over $Y$ of the same rank as $q$. Using the description of the cohomological set in terms of classes of cocycles,
it  is easily verified that
$\rho$ induces a natural map
$$\rho_{*}: \check H^{1}(Y, G)\rightarrow \check H^{1}(Y, {\bf O}(q))\ .$$
There is a canonical bijection between the isomorphism classes of $G$-torsors over $Y$ and the set $\check H^{1}(Y, G)$, [M], III,
Theorem 4.3 and Proposition 4.6. Therefore any $G$-torsor $X$ defines via $\rho_{*}$ a class in $ \check H^{1}(Y, {\bf O}(q))$ and
thus a symmetric bundle,  which is unique up to isometry. We shall call this new symmetric bundle {\it the twist of $(M, q)$ by $( \rho, X)$}.
One should observe that this definition coincides with   Fr\"ohlich's definition  when considering $Y=\mbox{Spec}(K)$,  $X=\mbox{Spec}(L)$
as the torsor and $G$ as  the constant group scheme associated with the Galois group $\mbox{Gal}(L/K)$.

Our aim is to generalize
Fr\"ohlich's work in this geometric framework. This then leads us to two different questions: first to provide an explicit description of the twist; and
then secondly to obtain comparison formulas between the invariants of the form and the invariants of its twists. The process of  twisting an equivariant
 symmetric bundle by an orthogonal representation provides us with a way to generate new symmetric bundles. It is indeed important
 to be able to identify these new bundles and to compare their invariants with the invariants of the initial form and the representation. In this paper we consider the first question: namely, to give an explicit formula for a twist, which leads itself to detailed calculations.  This is a matter that we  treated
 previously in  [CNET1]  and [CNET2] for constant groups schemes. In these papers, because  $G$ was a  constant group scheme,  it was possible to generalize  Fr\"ohlich's construction
in a relatively straightforward manner. Here though we wish to deal with equivariant symmetric bundles for the action of an arbitrary
  finite and flat group scheme; we are therefore obliged to modify
substantially  our previous constructions in order  to take into account the fact that the group scheme $G$ is no longer  \'etale  with the result  that the torsors will admit
 ramification. This then is a crucial and major difference with the previous constant case.

 In this paper we consider the affine situation for equivariant bundles for non-constant group schemes.  The general case will then follow by  patching  together the affine constructions. The natural framework for this article is essentially algebraic; our tools derive from the theory of Hopf algebras and principal homogeneous spaces, with the module of integrals of a Hopf order playing an exceptionally important role. To avoid any risk of confusion,  we have adhered, where possible, to  the
 geometric language used previously  in [CNET1] and [CNET2].

 The first  problem that we encounter  in the  geometric context, when seeking comparison formulas,  derives  from the definitions of the various cohomological invariants that we wish to
attach to our forms and representations.  Jardine in [J1], [J2], [J3] and later, Esnault, Kahn and Viehweg in  [EKV] and [K]
have considered  this question.
In our forthcoming  paper [CCMT] we will revisit their work with the language of topos. We believe that this approach casts important new light on
the subject and will  provide us with new and extremely  efficient tools
for the further generalizations that we  have in mind  (where we wish to consider  tame actions instead of torsors).
The equivalence of categories between  torsors and morphisms of topos, as described by Grothendieck in [SGA4], IV, Section 5,  will
play a key role in this coming paper. The constructions of this paper illustrate this fundamental point in the affine situation.

 We conclude this introduction by describing the contents and structure of the paper. In Section 2, after  some elementary algebraic results on Hopf algebras,
  we describe how to  attach to  any  equivariant symmetric bundle a further symmetric bundle obtained by taking fixed points. This
  construction  is crucial for the rest of the paper.  It  leads us to define (see Definition 3 of Section 3) the
  algebraic twist of $(M, q)$ by
  $( \rho, X)$. In Theorem 4.1 of Section 4, we verify that this form indeed coincides with the twist of the form
  introduced  previously. In Section 5 we use our methods to study examples of orthogonal representations of non-constant group schemes and the
 construction of twists.

\section{    Symmetric bundles and fixed points}

The goal  of this section is to describe how we  can associate  to any symmetric bundle,  equivariant under the action
of a finite and flat Hopf algebra,  a new symmetric bundle by taking {\it fixed points}. Prior to  describing  this procedure
in Subsection 2.2,
Proposition 2.9,   in the first subsection we have assembled
the main notation of the paper together with  some elementary algebraic results on Hopf algebras that we will use later on.
\vskip 0.1 truecm
\subsection {Algebraic preliminaries}

Let $R$ be a commutative noetherian integral domain in which $2$ is invertible,  with field of fractions  $K$.
 We consider  a finite,  locally free  $R$-Hopf algebra $A$ and   we denote by $A^{D}$ the
dual algebra.   We set  $A_{K}=A\otimes _{R}K$,  $ A_{K}^{D}=A^{D}\otimes _{R}K$ and  we
identify $A_{K}^{D}$ with the dual of $A_{K}$. We let $\Delta,\  \varepsilon $ and $S$ (resp. $\Delta^{D},\  \varepsilon^{D} $ and $S^{D}$) be respectively
the comultiplication, the counit and the antipode of $A$ (resp. $A^{D}$). We assume that $S^{2}=I_{A}$,  which implies that
$(S^{D})^{2}=I_{A^{D}}$. This last condition is fulfilled  when $A$ is commutative or cocommutative [C], Proposition 1.11.
 A right $A$-comodule $M$ is a finitely generated,  locally free $R$-module,  endowed with a structure map
$$\alpha _{M} : M\rightarrow M\otimes_{R} A \ $$
$$m\mapsto \sum _{(m)}m_{(0)}\otimes m_{(1)}\ .$$
Define the $R$-linear map
$$\psi _{M}: A^{D}\otimes_{R} M \rightarrow M$$
$$g\otimes m\mapsto \sum _{m} <g, m_{(1)}>m_{(0)} \ .$$
One can prove that $\psi _{M}$ defines a left $A^{D}$-module structure on $M$. Moreover, by Proposition 1.3 in [CEPT],   the association
$(M, \alpha _{M}) \rightarrow (M, \psi _{M})$ gives a bijective correspondence between the $A$-comodule and the $A^{D}$-module
structures on a $R$-module $M$. For any $A$-comodule $M$ we define  the $R$-sub-module
$$M^{A}= \{ m\in M \mid  \alpha _{M}(m)=m\otimes 1 \} \  . $$

\begin {lem} For any $A$-comodule $M$, then
$$M^{A}= \{m \in M \mid  gm=g(1_{A})m\ \    \forall g\in A^{D} \}\ .$$
\end {lem}
\begin{proof}Let $M'$ denote the right hand side of the equality.  The inclusion $M^{A}\subset M'$ is immediate.
We now use
the fact that the map
$$\varphi: A\otimes_{R}A^{D}\rightarrow \mbox{Hom}_{R}(A, A), $$
with $\varphi (h\otimes f))(a)=<f, a>h$ is an isomorphism. Therefore there exist elements $\{h_{1},...,h_{n}\}$ of $A$ and
$\{f_{1},...,f_{n}\}$ of $A^{D}$ such that
$$i_{d}=\sum_{1\leq i\leq n}\varphi (h_{i}\otimes f_{i})\ .$$ This implies that
$h=\sum_{1\leq i\leq n} <f_{i}, h>h_{i}$ for any $h\in A$.  It now follows from the definitions that,  for any
$m\in M$,  we  have:
$$\alpha_{M}(m)=\sum_{(m)}m_{0}\otimes m_{(1)}=\sum_{(m)}m_{(0)}\otimes (\sum_{i}<f_{i}, m_{(1)}>h_{i})$$
$$=\sum_{i, (m)}(<f_{i}, m_{(1)}>m_{(0)})\otimes h_{i}=\sum_{i}f_{i}m\otimes h_{i}\ . $$
In particular,  for any  $m\in M'$,  we obtain that
$$\alpha_{M}(m)=\sum_{i}f_{i}m\otimes h_{i}=\sum_{i}\varepsilon ^{D}(f_{i})m\otimes h_{i}
=m\otimes \sum_{i}\varepsilon ^{D}(f_{i})h_{i}=m\otimes 1_{A}\ .$$
Hence  we have proved that $M'\subset M^{A}$.
\end{proof}
\medskip

Since any Hopf algebra is  a  left module over itself via the multiplication map,  it also has a right comodule structure    on its
 dual. Therefore it follows from the lemma that we may define the left integrals of $A$ and $A^{D}$ by the following equalities:
 $$I(A)=A^{A^{D}}=\{x\in A \ \mid ax=\varepsilon (a)x, \forall a\in A\}, $$

 $$I(A^{D})=(A^{D})^{A}=\{f\in A^{D}\ \mid  uf=\varepsilon^{D} (u)f=u(1)f, \forall u\in A^{D}\}\ .$$
 We note  that $I(A)$ is not only an $R$-submodule of $A$ but also a two-sided $A$-ideal.
 In a similar way we  may define the module
 of right integrals. A Hopf algebra is called  {\it unimodular} if the modules of left and right integrals coincide.
  A Hopf algebra is  also endowed with a right comodule structure induced by its  comultiplication. Therefore it becomes  a left module over
  the  dual algebra as explained previously. The description of a finite Hopf $R$-algebra as a module over its dual holds in general.
  A  theorem of Larson and Sweedler ( see [Sw], Theorem 5.1.3) states
 that for any finite Hopf $R$-algebra the action of $A^{D}$ on $A$ induces an isomorphism
 $$A\simeq A^{D}\otimes_{R}I(A)\ .$$
  This theorem implies that $I(A)$ and $I(A^{D})$ are rank one projective $R$-modules, [C], Corollary 3.4. In the particular case
 where $I(A)$ is a free $R$-module with $\theta $ as a basis, then $A$ is a free $A^{D}$-module on the left integral $\theta$. This is
 always the case when $R$ is a principal ideal domain.
   \begin{lem}
The following properties  are equivalent:

\item i) The module of left integrals of $A$ is a free rank one $R$-module.

\item ii) The module of left integrals of $A^{D}$ is a free rank one $R$-module.

\item iii) There exists  $\theta\ \in A$ and $\theta^{D}\ \in A^{D}$ such that $\theta^{D} \theta =1_{A}.$
\end{lem}

\begin{proof} We  show  that $i) $ implies $ii)$.  The rest of the proof  is left to the reader.
Let $\theta $ be a basis of $I=I(A)$. Since  $A$ is a free $A^{D}$-module on  $\theta$ there exists a unique $\theta^{D}$ in
$A^{D}$ such that $1_{A}=\theta^{D}\theta$.  For any $u$ of $A^{D}$ we have the equalities:
$$(u\theta^{D})\theta=u(\theta^{D}\theta)=u1_{A}=\varepsilon ^{D}(u)1_{A}=(\varepsilon^{D} (u)\theta^{D})\theta\ .$$
This implies  that $u\theta^{D}=\varepsilon^{D}(u)\theta^{D}$ and hence  that $\theta^{D}$ is a left integral of $A^{D}$.
Let $u$ be a non-zero left integral of $A^{D}$. Since $A^{D}$ is a projective
 $R$-module, it follows that $I^{D}$ is contained in $I^{D}_{K}=I(A^{D}_{K})$, which is a $K$-vector space of dimension one.
 Therefore there exist non-zero elements $m$ and $n$ of $R$ such that $mu=n\theta^{D}$. We set $t=u\theta$. This is an element of $A$.
We observe that
$$mt=(mu)\theta=n(\theta^{D}\theta)=n\ .$$
It follows that $n=m\varepsilon (t)$, and so $m$ divides $n$ in $R$, and $u$ is a multiple of $\theta^{D}$. We conclude that
$\theta^{D}$ is a free generator of $I(A^{D})$.
\end{proof}

\begin {lem} The following properties are equivalent:

\item i) There exists  $\lambda \in \{\pm 1\}$ and a non-zero element $a$ of $I(A)$ such that
$$S(a)=\lambda a\ . $$

\item ii) There exists $\lambda \in \{\pm 1\}$ such that

$$S(x)=\lambda x, \ \forall x\in I(A)\ . $$

\item iii) $A$ is unimodular.
\smallskip
\vskip 0.1 truecm

\noindent Moreover if $A_{K}$ is separable,  then $S(x)= x, \ \forall x\in I(A)$.

\end {lem}
\begin{proof}The equivalence between $i)$ and $ii) $ follows from the fact that  $I(A_{K})$ is a $K$-vector space of dimension $1$.
 Let $I'(A)$ be the set of right integrals of $A$. For any element $u\in A$ and $x\in I(A)$ we know that
$ux=\varepsilon (u)x$. This implies that  $S(ux)=S(x)S(u)=\varepsilon (u)S(x)$. Since $S$ is bijective we deduce that
$$I'(A)=\{S(x),\ x\in I(A) \}\ . $$ Therefore $ii)$ implies $iii)$. We now assume $iii)$. Let $a$ be a non-zero element of $I(A)$. Since
$a$ and $S(a)$ both belong to $I(A)$ and since $I(A_{K})$ is of dimension one  there exists $\lambda \in K$ such that
$S(a)=\lambda a$. Using that $S^{2}=I$ we deduce that $\lambda \in \{\pm 1\}$ which proves $i)$. Finally suppose that $A_{K}$ is
semi-simple. Our aim is to prove that $S(x)=x$ for any $x \in I(A)$. Since  $I(A)\subset I(A)\otimes _{R}K=I(A_{K})\ $   it suffices to prove that $S(x)=x$  for any $x$ in $I(A_{K})$. It
 follows from [Sw], Theorem 5.1.8,  that $A_{K}=I(A_{K})\oplus \mbox{Ker}(\varepsilon )$ as a direct sum of
 $A_{K}$-ideals. Let   $x$ be a non-zero element of $I(A_{K})$. Then $S(x)$ can be decomposed as a sum $rx+y$ with  $r\in K$ and
 $y\in \mbox{Ker}(\varepsilon )$.  Since $ \varepsilon\circ S=\varepsilon$ we deduce that $r=1$. Therefore
 $S(x)x=x^{2}+yx$. We observe that  $yx\in I(A)\cap \mbox{Ker}(\varepsilon )$. Thus  $yx=0$ and
 $S(x)x=\varepsilon(x)x=\varepsilon(x) S(x)$. We conclude that $x=S(x)$ as required.
 \end{proof}

\vskip 0.1 truecm

\noindent {\bf Definition 1}   {\it A Hopf algebra satisfies hypotheses ${\bf H_{1}}$ when  it is unimodular,
 when its module of integrals  is a free $R$-module and when the restriction of $S$ to the module of integrals is the identity map. }

\vskip 0.1 truecm
\noindent {\bf Remark } When $K$ is of characteristic $0$  and $A_{K}$ is commutative, it  follows from a theorem of Cartier that $A_{K}$
 is separable. Hence in this case we deduce from Lemma 2.3 that $A_{K}$ and therefore $A$ is unimodular and that
$S(x)=x$ for any $x\ \in I(A)$. Therefore any finite, commutative and locally free $R$-Hopf algebra, where $R$ is a
principal ideal domain of characteristic $0$, satisfies ${\bf H_{1} }$.  Nevertheless it is easy to construct Hopf algebras which are not separable
but which satisfy ${\bf H_{1} }$.  It suffices for instance to consider $A=k[\Gamma ]$ where $k$ is a field of characteristic $p$ and
$\Gamma $  a finite group of order divisible by $p$,  endowed with its usual Hopf algebra structure. We know  from Maschke's Theorem
that $A$ is not separable. However, we may easily  check that $I(A)=k\omega $ where $\omega =
\sum_{\gamma \in \Gamma}\gamma$. Since $S(\omega )=\omega $ we deduce that $A$ satisfies ${\bf H_{1}}$. When $A_{K}$ is
not separable, $I(A_{k})$ is contained in $\mbox{Ker}(\varepsilon)$ and thus $I(A_{k})^{2}=I(A)^{2}=\{0\}$.
\begin {lem} Assume that $A$ is  commutative,  $I(A)$ is free over $R$ and $A_{K}$ is separable. Then $A$ and $A^{D}$
both satisfy ${\bf H_{1}}$.

\end {lem}
\begin{proof} It follows immediately from Lemma 2.3 that $A$ satisfies ${\bf H_{1}}$ and thus,  from Lemma 2.2,  that $I(A^{D})$ is
$R$-free.  Since
$A_{K}$ is separable there exists  a finite extension $K'/K$ such that   $A_{K'}=A_{K}\otimes_{K}K'$ is the  algebra
$\mbox {Map}(\Gamma , K')$, where $\Gamma $ is a finite  group, endowed with its natural structure of Hopf algebra.
Since $K'$ is flat over $R$ and $A^{D}$ is finitely generated,  the  map
$$\alpha _{K'}: A^{D}\otimes_{R}K'\simeq A_{K'}^{D}$$

$$f\otimes \lambda\rightarrow  f\lambda$$
is an isomorphism of $K'$-vector spaces. We may easily check that it respects the algebra and coalgebra structure of both sides. Thus
we may identify the Hopf algebras $ A^{D}\otimes_{R}K'$ and $K'[\Gamma ]$ and therefore  $I(A^{D})$
with $I(K'[\Gamma ])\cap A^{D}$.
The unimodularity of $A^{D}$ follows from the unimodularity of $K'[\Gamma ]$. We now want to prove that the restriction of $S^{D}$ to
$I(A^{D})$ is the identity map. Let $\theta^{D}\in A^{D}$ and $\theta \in A$ such that $\theta^{D}\theta= 1_{A}$.
Let $\Delta (\theta)=\sum \theta_{(0)}\otimes \theta_{(1)}$. We deduce from the definitions that
$$\varepsilon (\theta^{D}\theta)=1_{R}=\sum \varepsilon (\theta_{0})<\theta^{D}, \theta_{(1)}>=<\theta^{D}, \theta>\ $$
and
$$\varepsilon (S^{D}(\theta^{D})\theta)=<\theta^{D}, S(\theta)>=<\theta^{D}, \theta>\ . $$ Because $A^{D}$ is unimodular
we know by Lemma 2.3 that $S^{D}(\theta^{D})$ is either $\theta^{D}$ or $-\theta^{D}$. Since the characteristic of $K$ is different from $2$
we deduce from the previous equalities that $S^{D}(\theta^{D})=\theta^{D}$. We have proved,  as required,  that $A^{D}$ satisfies
${\bf H_{1}}$.
\end{proof}
\vskip 0.1 truecm

Let $M$ be an $A$-comodule and  let  $M_{A}$  be the largest quotient of $M$ on which
$A^{D}$ acts trivially, so that $M_{A}=M/ker(\varepsilon^D)M$.

\begin{lem} \label{fixedtrace}
Suppose  that $A^{D}$ satisfies ${\bf H_{1}}$ and that  $\theta^{D}$ denotes an  $R$-basis of $I(A^{D})$. Let  $M$ be  a projective $A^{D}$-module. Then

\begin{itemize}
\item[i)]  $M^{A}=\theta^{D}M$,
\item[ii)]$ M^{A}$ is a locally free $R$-module,
\item[iii)] The map $m \mapsto \theta^{D}m $ induces an isomorphism of $R$-modules from $M_A$ onto $M^{A}$.

\end{itemize}

\end{lem}

\begin{proof} We first observe that we can reduce to the case where $M$ is $A^{D}$-free and so  is a direct sum of copies of
$A^{D}$. Therefore, in order  to prove  the lemma, we may  assume that  $M=A^{D}$.
In this case it follows from the very definition of the set of integrals  that
$M^{A}=(A^{D})^{A}=I(A^{D})=\theta^{D}A^{D}=\theta^{D}R$ which proves i) and ii)  of the lemma.
Moreover,  for $g\in A^{D}$,  the equality $\theta^{D}g=0$ is equivalent to $\theta^{D}g(1)=0$ and thus to $g(1)=0$ since
$A^{D}$ is $R$-torsion free. We then deduce that the kernel of the $R$-module homomorphism  $m\mapsto \theta^{D}m$ is
the submodule $Ker(\varepsilon^{D})M$. Therefore it induces,  as required,  an isomorphism from $M_{A}$ onto $M^{A}$.
\end{proof}
\vskip 0.1 truecm
Let $(M, \alpha _{M})$ and $(N, \alpha _{N})$ be $A$-comodules. We  shall  define a comodule structure on $M\otimes N$ by
considering
$$\alpha_{M, N} :M\otimes N \stackrel{\alpha_{M}\otimes \alpha_{N}}{\longrightarrow }  M\otimes A\otimes N\otimes A
\simeq M\otimes N\otimes A\otimes A\stackrel{Id\otimes mult}{\longrightarrow} M\otimes N\otimes A  \ .$$
The $A^{D}$-module structure associated to this comodule structure is given by:
$$g (m\otimes n)= \sum_{(g)}g_{(0)}m\otimes g_{(1)}n \ \  , \forall g\in A^{D}, m\in M, n\in N\ , $$
where $\Delta _{D}(g)=\sum_{(g)}g_{(0)}\otimes g_{(1)}$. We  say that {\it $M\otimes N$ is endowed with the diagonal action of
$A^{D}$}.
\vskip 0.1 truecm

We conclude  this section by proving a result  which generalizes to a large family of    Hopf algebras a theorem well known  when
$A=\mbox {Map}(\Gamma , R)$ and $A^{D}=R[\Gamma ]$, [Mc] Corollary 3.3,
p.145 and p.196.

\begin{prop} \label{again}
Let $A$ be   a Hopf algebra over $R$.   Let
$M$ and $N$ be $A^{D}$-modules. Assume that $M$ and $N$ are both projective $R$-modules and that  either $M$ or $N$ is
projective as an $A^{D}$-module. Then $M\otimes_R N$ endowed with the diagonal action of $A^{D}$  is a projective $A^{D}$-module.
\end{prop}
\begin{proof}

Observe that for any $A^{D}$-modules $M$ and $N$, then 
$\mbox{Hom}_{R}(M, N)$ is endowed with the structure of an $A^{D}$-module by the rule:
 $$<g\varphi, m>=\sum _{(g)}g_{(0)}<\varphi ,S^{D}(g_{(1)})m >\ \ \forall g\in A^{D}$$
 where $\Delta ^{D}(g)= \sum _{(g)}g_{(0)}\otimes g_{(1)}$.
 
 \begin {lem}  For any $A^{D}$-modules $M$ and $N$,   one has the equality: 
  $$ \mbox{Hom}_{A^{D}}(M, N)= \mbox{Hom}_{R}(M, N)^{A} . $$ 

 \end {lem}
\begin{proof}
Let $\varphi$ be an element of 
 $\mbox{Hom}_{A^{D}}(M, N)$. Then,  for all
 $g\in A^{D}$,  we
 have
 $$
 <g\varphi, m>=\sum _{(g)}g_{(0)}<\varphi, S^{D}(g_{(1)})m>= \sum _{(g)}g_{(0)}S^{D}(g_{(1)})<\varphi, m> \ .
 $$
 Therefore
 $$<g\varphi, m> =\varepsilon^{D}(g)<\varphi ,m>\ . $$
This implies that 
 $$\mbox{Hom}_{A^{D}}(M, N)\subset \mbox{Hom}_{R}(M, N)^{A} .$$
 We want to prove that this inclusion is an equality. Let  $\varphi$ be an element of $\mbox{Hom}_{R}(M, N)^{A}$. Then,  for any 
 $a\in A^{D}$ and $m\in M$, we observe that since 
 $$a=\sum \varepsilon^{D}(a_{(0)})a_{(1)}, \ \mbox{with}\  \Delta^{D}(a)=\sum_{(a)}a_{(0)}\otimes a_{(1)}, $$ 
 we may write 
 $$<\varphi, am>=\sum_{(a)}\varepsilon^{D}(a_{(0)})<\varphi, a_{(1)}m>. $$
 Using that $\varphi \in \mbox{Hom}_{R}(M, N)^{A}$ we obtain that 
 $$<\varphi, am>=\sum_{(a)}<a_{(0)}\varphi, a_{(1)}m>.$$

 Hence we conclude that 
 $$<\varphi, am>=\sum_{(a)}a_{(3)}<\varphi, S^{D}(a_{(4)})a_{(1)}m>  $$
 where $(\Delta^{D}\otimes id )\circ \Delta^{D}(a)=\sum a_{(3)}\otimes a_{(4)}\otimes a_{(1)}$. 
 
 We now observe that the map 
 $$ A^{D}\times A^{D}\times A^{D}\rightarrow N$$
  $$ (a,b,c)\mapsto a<\varphi, S^{D}(b)cm)> $$ 
  is trilinear and therefore induces 
  $$\Theta: A^{D}\otimes A^{D}\otimes A^{D}\rightarrow N.$$
   It follows from the coassociativity of $A^D$ that we have the equality in $A^D\otimes_{R} A^D\otimes_{R} A^D$:
 $$(id\otimes \Delta^{D})\circ \Delta^{D}(a)=(\Delta^{D}\otimes id )\circ \Delta^{D}(a).$$
  Denoting  by $b$ this element,  then  we can write 
 $$\Theta (b)=<\varphi, am> . $$
 For any $a_{(1)}$  we set $\Delta^{D} (a_{(1)})=\sum a'_{(1)}\otimes a''_{(1)} $. It follows that  
  $$b= (id\otimes \Delta^{D})\circ \Delta^{D}(a)=\sum_{(a)}a_{(0)}\otimes a'_{(1)}\otimes a''_{(1)}.$$
  Therefore we have the equality:  
  $$\Theta(b)=\sum_{(a)}a_{(0)}<\varphi, S^{D}(a'_{(1)})a''_{(1)}m>, $$
  which can be written  
  $$\Theta(b)=\sum a_{(0)}<\varphi, (\sum S^{D}(a'_{(1)})a''_{(1)})m>. $$
  Since we know that $\varepsilon^{D}(a_{(1)})=\sum S^{D}(a'_{(1)})a''_{(1)}$ and belongs to $R$  for any $a_{(1)}$, we 
  deduce that 
  $$\Theta (b)=(\sum_{(a)}a_{(0)}\varepsilon^{D}(a_{(1)}))<\varphi, m>=a<\varphi, m>. $$
  Comparing the two expressions for  $\Theta(b)$ we deduce that 
  $$<\varphi, am>=a<\varphi, m>, \ \forall a\in A^{D}\ \mbox{and}\ \forall m\in M. $$
  The required inclusion then follows. This achieves the proof of the lemma. 
  \end{proof}
  
 The proof of the proposition now follows closely the proof of Corollary 3.3 in Chapter V.3 of [Mc]. We may assume for instance that $M$ is a  projective
$A^{D}$-module.
  We have to prove  that the functor $P \rightarrow \mbox{Hom}_{A^{D}}(M\otimes_R N,P)$
from the category of $A^{D}$-modules into the category of $R$-modules is exact.  By the adjoint property  of $\mbox{Hom}$ and $\otimes$ we know that  for any $R$-module $P$ we have a natural isomorphism of $R$-modules
$$ \psi_P:  Hom_R(M\otimes_R N,P) \simeq Hom_R(M,Hom_R(N,P))$$
$$f \mapsto\psi_{P}(f):m \mapsto (n \mapsto f(m\otimes n)).\ $$
We assert that if $P$ is an $A^{D}$-module then the adjoint isomorphism restricts to an isomorphism of $R$-modules 

 $$ \psi'_P:  Hom_{A^D}(M\otimes_R N,P) \simeq Hom_{A^D}(M,Hom_R(N,P)).$$
 Since the functors $P \rightarrow \mbox{Hom}_{R}(N,P)$ and $Q \rightarrow \mbox{Hom}_{A^{D}}(M,Q)$ are exact the result will then follow.

 We now want to prove the assertion. We start by proving that $\psi_{P}$ is an isomorphism of $A^{D}$-modules. 
 Let $\varphi$ be an element of $Hom_R(M\otimes_R N,P)$. Then for any for any $a\in A^D, m\in M, n \in N$  on the one hand we have
$$
 <\psi_{P}(a\varphi)(m), n>=<a\varphi, m\otimes n>=\sum _{(a)}a_{(0)}<\varphi, S^{D}(a_{(3)}) m\otimes S^{D}(a_{(2)})n>$$ 
 
where $(id\otimes \Delta^{D})\circ \Delta^{D}(a)=\sum a_{(0)}\otimes a_{(1)}\otimes a_{(2)}$.

On the other hand
 $$a\psi_{P}(\varphi)(m)=\sum _{(a)}a_{(0)}\psi_{P}(\varphi)(S^{D}(a_{(1)})m)$$
 and thus
 $$<a\psi_{P}(\varphi)(m), n>= \sum _{(a)}a_{(3)}<\psi_{P}(\varphi)(S^{D}(a_{(1)})m), S^{D}(a_{(4)}n> $$
 
$$ =\sum _{(a)}
 a_{(3)}<\varphi, (S^{D}(a_{(1)})m\otimes (S^{D}(a_{(4)})n>$$

 where $(\Delta^{D}\otimes id )\circ \Delta^{D}(a)=\sum a_{(3)}\otimes a_{(4)}\otimes a_{(1)}$.

 We now consider the map
 $$ A^D\times A^D\times A^D\rightarrow P$$
 $$ (x,y,z)\mapsto x<\varphi, S^{D}(z)m)\otimes S^{D}(y)n>.$$
 This is a trilinear $R$-map and it therefore induces 
 $$\Phi:A^D\otimes_{R} A^D\otimes_{R} A^D\rightarrow P.$$
 It follows once again from the coassociativity of $A^D$ that 
 $$(id\otimes \Delta^{D})\circ \Delta^{D}(a)=(\Delta^{D}\otimes id )\circ \Delta^{D}(a).$$
 Let us denote the above element by $c$.  It follows from the previous equalities that 
 $$ \Phi(c)=<\psi_{P}(a\varphi)(m),n>=<a\psi_{P}(\varphi)(m),n> .$$ 
 Hence we have proved as desired that 
 $$\psi_{P}(a\varphi)=a\psi_{P}(\varphi), \ \forall a \in A^{D}\ \mbox{and}\ \varphi\in Hom_R(M\otimes_R N,P).$$

 We    achieve the proof of the assertion. It is easily verified that since $\psi_{P}$ is an $A^{D}$-isomorphism it 
 induces an isomorphism: 
 $$\psi'_{P}: Hom_R(M\otimes_R N,P)^{A} \simeq Hom_R(M,Hom_R(N,P))^{A} .$$
 It now follows from  Lemma 2.7  that this is precisely the required isomorphism. 
 \end{proof} 

\subsection {Equivariant symmetric bundles}

\noindent A symmetric bundle  over $R$ is a
 finitely generated,  locally free $R$-module $M$, equipped
with a non-degenerate,  bilinear,  symmetric form
$$q: M\times M\rightarrow R . $$
Let $A$ be a finite and locally free Hopf algebra over $R$ and $(M, q)$ be a symmetric bundle over $R$. We shall say that  $(M, q)$ is
{\it $A$-equivariant} if   $M$ is  an $A$-comodule and if the following is true:

$$ \ q(gm,n)=q(m,S^{D}(g)n), \ \ \forall   m, n \in M,  \ \forall   g\in A^{D}\ \ \ . $$
If, moreover,    $M$ is a projective $A^{D}$-module,  we shall call $(M, q)$ a {\it projective $A$-equivariant bundle}.
Note that when $A=\mbox{Map} (\Gamma , R)$,  with $\Gamma $  a finite group
and  $A^{D}$ is the group algebra $R[\Gamma ]$,    an $A$-equivariant  symmetric bundle  is an
$R[\Gamma ]$-module  endowed with a non-degenerate, $\Gamma $-invariant, bilinear and symmetric form.

We observe that any $A$-equivariant symmetric bundle $(M, q)$ defines, after scalar extension by a commutative $R$-algebra $T$, an
$A_{T}$-equivariant symmetric bundle over $T$ that we denote by $(M_{T}, q_{T})$.

We can attach to any $R$-symmetric bundle  $(M, q)$ its orthogonal group that we denote by ${\bf O}(q)$. This is an affine group scheme
over $R$. This group scheme is most easily defined in terms of its associated functor of points. It is the functor which sends an $R$-algebra
$T$ to the orthogonal group of the form $(M_{T}, q_{T})$. Suppose now that $A$ is a commutative Hopf algebra.  Then we can associate to
$A$ the group scheme $G=\mbox{Spec}(A)$. We will say that $G$ is generically \'etale when $A_{K}$ is a separable $K$-algebra.
 In this case  the notion of $A$-equivariant
symmetric bundle has an interpretation  in terms of orthogonal representations.

\begin {prop} Let $A$ be a commutative Hopf $R$-algebra and let $G$ be the group scheme defined by $A$. We assume that $G$
is generically \'etale. Let $(M, q)$ be an $R$-symmetric bundle. Then the following properties are equivalent:
\item i) $(M, q)$ is $A$-equivariant.

\item ii) There exists a morphism of group schemes:
$$\rho: G\rightarrow {\bf O}(q).$$
\end{prop}
\begin{proof} For any $R$-algebras $E$ and $F$ we denote by $\mbox {Hom}_{R, alg}(E, F)$ the set of morphisms of $R$-algebras
$f: E\rightarrow F$. We recall that for any $R$-algebra $T$ we have the following isomorphism
$$G(T)\simeq \mbox{Hom}_{R, alg}(A, T)\simeq \mbox{Hom}_{T, alg}(A_{T}, T)\ .$$
Moreover, since $A$ is a finitely generated projective $R$-module, we know that
$$\mbox{Hom}_{T}(A_{T}, T)\simeq \mbox{Hom}_{R}(A, R)\otimes_{R}T \ .$$
With a slight abuse of notation we write
$$G(T)=\mbox{Hom}_{T, alg}(A_{T}, T )\subset \mbox{Hom}_{R}(A, R)\otimes_{R}T\ .$$
We assume $i)$. Since $M$ is an $A^{D}$-module, after scalar extension $M_{T}$ becomes a $A^{D}_{T}$-module. Therefore
for any $g\in G(T)$ we may define
$$\rho_{T}(g): M_{T}\rightarrow M_{T}, \quad m\mapsto gm\ .$$
 One checks that $\rho_{T}$ induces a group homomorphism
from $G(T)$ to $\mbox{Aut}(M_{T})$. We now observe that, since $q_{T}$ is $A^{D}_{T}$-equivariant, for all $ m, n \in M_{T}$
$$q_{T}(\rho_{T}(g)(m), n)=q_{T}(gm, n)=q_{T}(m, S^{D}(g)n)=q_{T}(n, \rho_{T}(g)^{-1}(n)), $$
(for the last equality we use that  the inverse of $g$ in $G(T)$ is $S^{D}(g)$). Therefore we conclude that
$\rho_{T}(g) \in {\bf O}(q_{T}))$ for any $g$, which proves $ii)$.
We now suppose that $ii)$ is satisfied. It follows from the hypothesis that there exists a group homomorphism
$$\rho_{A}: G(A)\rightarrow \mbox{Aut}(M_{A})\ .$$
For the element $i_{d}$ we obtain an $A$-linear map $\rho_{A}(i_{d}): M\otimes_{R}A\rightarrow M\otimes_{R}A$
which determines by restriction $\alpha: M\rightarrow M\otimes_{R}A$. The map $\alpha$ endows $M$
with a right $A$-comodule structure (see [W], Chapter 3) and therefore with a left $A^{D}$-module structure. Since there exists a
finite extension $K'/K$ such that
$A^{D}_{K'}=K'[\Gamma]$, any element $g\in A^{D}$ can be written $g=\sum_{\gamma\in \Gamma}r_{\gamma}\gamma$
with $r_{\gamma}\in K',\  \forall \gamma\in \Gamma$. Since every $\gamma$ belongs to $G(K')$, then $S^{D}(\gamma)=\gamma^{-1}$.
Since $\rho_{K'}(\gamma)$ belongs to ${\bf O}(q_{K'})$ for any $\gamma \in \Gamma$ we can write the equalities:
$$q(gm, n)=\sum_{\gamma\in \Gamma}r_{\gamma}q(\gamma m, n)=
\sum_{\gamma\in \Gamma}r_{\gamma}q( m, \gamma^{-1}n)=q(m, S^{D}(g)n), $$
for any $m$ and $n$ $\in M$ and $g\in A^{D}$. Hence we have proved as required that $(M, q)$ is $A$-equivariant.
\end{proof}
\medskip
\vskip 0.1 truecm
  We  call any morphism of group schemes $\rho:
 G\rightarrow {\bf O}(q)$ an {\it orthogonal representation of  $G$}. In this paper we will frequently speak either of equivariant symmetric bundles or orthogonal representations.
 One should notice that when $G$ is generically constant an orthogonal representation,  as defined above,  induces by restriction to the
 generic fiber an orthogonal representation in the usual sense.

 \vskip 0.1 truecm

Let $(M, q)$ be a projective  $A$-equivariant symmetric bundle with $A^{D}$ satisfying ${\bf H_{1}}$.   We fix
an $R$-basis $\theta^{D}$ of $I(A^{D})$.  Under these assumptions we use  Lemma 2.5  to define a map
$$q^{A}:M^{A}\times M^{A}\rightarrow R$$
by setting
$$q^{A}(x, y)=q(m, y)=q(x, n)$$
where $m$ (resp. $n$) is any arbitrary element of $M$ such that $x=\theta^{D}m$  (resp. $y=\theta^{D}n$).
Observe that if $\theta^{D}m=\theta^{D}m'$, then $m-m'$ belongs to $\mbox {Ker} (\varepsilon ^{D})M$.  Since
$$q(gu, \theta^{D}n)=q(u, S^{D}(g)\theta^{D}n)=q(u, \varepsilon ^{D}(S^{D}(g))\theta ^D n)=
q(u, \varepsilon ^{D}(g)\theta ^D n)=0,  $$
for any $g\in \mbox {Ker} \varepsilon ^{D}$,   we deduce that $q^{A}$ is well defined. Moreover, using Lemma 2.3,  we note that
$$q^{A}(x, y)=q(y, m)=q(\theta^{D}n, m)=q(n, S^{D}(\theta^{D})m)=q(n, \theta^{D}m)=q^{A}(y, x) \ . $$
Hence $q^{A}$ is a bilinear symmetric form on $M^{A}$.

\begin{prop} \label{CohAl}
 Let $A$ be a Hopf algebra such that $A^{D}$ satisfies ${\bf H_{1}}$ and let $(M, q)$ be
 a projective $A$-equivariant symmetric bundle. Then $(M^A, q^A)$ is a symmetric
 $R$-bundle.
\end{prop}
\begin{proof} From Lemma 1.5 we know that $M^{A}$ is a locally free $R$-module. It remains to prove that the adjoint map
$$\varphi_{q^{A}}: M^{A} \rightarrow \mbox {Hom}(M^{A}, R)$$ is an
$R$-module isomorphism.  This result, when $A=\mbox {Map}(\Gamma , R)$ and  $\Gamma $ is
a finite group, was  proved in Proposition 2.2 of [CNET] .  This proof can   be used {\it mutatis mutandis} in this more general situation if,
as in the situation  considered in [CNET],  the quotient $M/M^{A}$ is torsion free. This is easily checked. We note that  it suffices to
prove the result when $M= A^{D}$. Let $f\in A^{D}$ and $d \in R$, $d\not\neq 0$ such that $df \in (A^{D})^{A}$.
It follows from the  definition of  $(A^{D})^{A}$  that for any $g \in A^{D}$ then $g(df)=\varepsilon ^{D}(g)(df)$. Since $A^{D}$ is a
projective $R$-module it is torsion free  and thus    $gf=\varepsilon ^{D}(g)f$,  which proves that  $f\in (A^{D})^{A}$.
\end{proof}

\medskip

\vskip 0.1 truecm
\noindent {\bf Remarks }

 \noindent {\bf 1.} For any $x=\theta^{D}m$ and $y=\theta^{D}n$ of $M^{A}$  it is easily verified that
 $$q(x, y)=q(m, S^{D}(\theta^{D})\theta^{D}n)=q(m, \varepsilon ^{D}(\theta^{D})y)=
 \varepsilon ^{D}(\theta^{D})q^{A}(x, y)\ .$$
 If  $A^{D}_{K}$ is not separable, then  we know that $\varepsilon^{D}(x)=0$ for any $x\in I(A^{D}_{K})$, from Theorem
 5.1.8 in [S]. Therefore this situation makes clear  that $q^{A}$ is not in general the restriction of $q$ to $M^{A}$,
 since, as we deduce from the previous equalities,  $q^{A}$ is unimodular while the restriction of $q$ to $M^{A}$ is zero.
\vskip 0.1 truecm
 \noindent {\bf 2.} It is important to note that the form $q^{A}$ depends upon  the choice of a generator of $I(A^{D})$.
 Taking  $\theta'^{D} =\lambda \theta^{D} $
 with $\lambda \in R^{*}$ as a new generator of $I(A^{D})$ provides us with a new symmetric form  $q'^{A}=\lambda q^{A}$
  on $M^{A}$. If $\lambda $ is a square of a unit of  $R$, then the symmetric bundles $(M^{A}, q^{A})$ and $(M^{A}, q'^{A})$ are isometric.
   As we will see, at the end of Section 2, our future constructions will not depend upon this choice.

\section {Twists of   symmetric bundles}

Recall that $R$ is an integral domain with field of fractions $K$ and that $A$ is a Hopf order in the Hopf algebra $A_{K}$. The
aim of this section is to define the {\it algebraic twist } of an $A$-equivariant symmetric bundle by a principal homogeneous space for $A$.
As a first step we show, under some assumptions on $A$, how to associate to a principal homogeneous space for $A$ an $A$-
equivariant projective vector bundle. The trace form is the key tool of this construction.

We let $A$ be a commutative Hopf algebra which is finite and flat over $R$.   Let  $B$ be  a commutative finite flat $R$-algebra, endowed with a structure of an $A$-comodule algebra
$$\alpha_{B}: B\rightarrow B\otimes_{R}A\ . $$  We suppose that $B^{A}=R$.
We shall say that  $B$ is {\it  a principal homogeneous space for $A$ over $R$}, abbreviated by PHS,  when
  $$(I_{d}\otimes 1, \alpha_{B }): B\otimes_{R}B\simeq B\otimes_{R}A\ $$  is  an isomorphism of $B$-algebras and
  left $A^{D}$-modules.  We observe that $A$,  endowed
with the comultiplication map,  provides  an example of such a  space.

\begin{lem} Let $A_{K}$ be a separable commutative Hopf algebra and let $B_{K}$ be a principal homogeneous space for $A_{K}$.
  Let $Tr$ be the trace form on $B_{K}$. Then $(B_{K}, Tr)$ is
a projective  $A_{K}$-equivariant symmetric bundle.

\end{lem}
\begin{proof} Since $B_{K}$ is a principal homogeneous space for $A_{K }$, we know  that $B_{K}$ is a projective $A^{D}_{K}$-module. Using the fact that $B_{K}$ becomes isomorphic to $A_{K}$ after a faithful base change, it follows by descent theory
that $B_{K}$ is separable and therefore that the trace is non-degenerate. We now want to show that the trace is an $A_{K}$-equivariant form.
As in Lemma 2.4  we fix a finite  extension $K'/K$ such that $A^{D}_{K'}$ is isomorphic to $K'[\Gamma ]$, where $\Gamma $ is
a finite group. In this case  $A^{D}_{K'}$,  as a $K'$-vector space,  has a basis $\{\gamma, \gamma\in \Gamma  \} $ consisting
of group like elements. Since $B_{K'}$ is an $A^{D}_{K'}$-module algebra, one easily checks that every $\gamma$ defines an automorphism of $K'$-algebra of
$B_{K'}$ whose inverse is $S^{D}(\gamma)$. Therefore the trace form $q_{K'}$ of $B_{K'}$ is invariant under each $\gamma \in \Gamma$.
Let $q$  denote  the trace form on $B_{K}$. Any element  $g$ of
  $A^{D}_{K}$ may be written  $g=\sum_{\gamma\in \Gamma }x_{\gamma}\gamma $  on the basis
$\{\gamma, \gamma\in \Gamma  \} $ of $A^{D}_{K'}$ with $x_{\gamma} \in K' $. Then,  for any element $x$ and $y$ of $B_{K}$,  we have the equalities:
$$q(gx, y)=\sum_{\gamma\in \Gamma }x_{\gamma}q_{K'}(\gamma x, y)=\sum_{\gamma\in \Gamma }x_{\gamma }q_{K'}(x, S^{D}(\gamma)y)
=q(x, S^{D}(g)y)\ . $$ We conclude that $q$ is $A_{K}$-equivariant, as required.
\end{proof}

We now wish to generalize the above construction when working with the ring $R$ in place of the field $K$.
 A key-role in this case  is played by  the codifferent of $B$.

\subsection { The   square   root   of   the   codifferent}

 The codifferent of $B$ is defined by
$$\mathcal{D}^{-1}(B) =\{x\in B_{K}\  \mid Tr(xb)\in R \ \ \forall b\in B\}   \ . $$
In the case where $R$ is a field then $\mathcal{D}^{-1}(B)=B$. It follows from the $A_{K}$-invariance of the trace form proved in
Lemma 2.1 that $\mathcal{D}^{-1}(B)$ is an $A^{D}$-module. We start by studying the codifferent of $A$.

\begin{prop}  Let $A$ be a commutative Hopf algebra, let  $I$ be the set of integrals of $A$, and assume that $A_{K}$ is separable. Then

\begin{itemize}
\item[i)]   There exists a unique primitive idempotent $e $ of $A_{K}$ and a fractional ideal $\Lambda $ of $R$ such that

$$I_{K}= Ke , \ \mbox {and}\ \  I=\Lambda  e \ \ .$$

\item [ii)] We have the equality:
$$\mathcal{D}^{-1}(A)=\Lambda^{-1} A$$

\end{itemize}

\end{prop}

\begin{proof} We know from [SW], Corollary 5.1.6.,  that $I_{K}$ is a one dimensional $K$-vector space. Moreover, since  $A_{K}$ is a
separable algebra,  we deduce from [SW], Theorem 5.1.8.  that $A_{K}= I_{K}\oplus {\mbox{Ker}( \varepsilon)} $. Let $u $ be a basis of
$I_{K}$.  Since
it is an integral, it follows that  $u ^{2}=\varepsilon (u)u$. Since $\varepsilon (I_{K})$ is non-zero, it follows that
$\varepsilon (u)$ is $\not\neq 0$. Therefore, replacing  $u$ by $u/\varepsilon (u)$, we obtain a new basis of $I_{K}$
which is a non-trivial idempotent.  We denote this  idempotent  by $e$. Since $R$ is an integral
domain, it follows that $\varepsilon (f)=1$ for any idempotent  $f$ of $I_{K}$. We conclude that $e$ is the unique idempotent of $I_{K}$.
Let $\Lambda$ be the fractional ideal of $R$ consisting of elements $x\in K$ such that $xe $ belongs to $A$. Then we have
$$I=I_{K}\cap A= \Lambda e\ \ .$$

\noindent We consider the  left $A_{K}^{D}$-module structure on $A_{K}^{D}$ defined by
$$<f*g, x>=<g, S^{D}(f)x>\ \ \forall x\ \in A_{K}\ \ .$$
Since $A_{K}$ is separable the trace form is non-degenerate and induces an isomorphism of $K$-vector spaces
$$\Psi : A_{K}\rightarrow A_{K}^{D}= \mbox {Hom}(A_{K}, K)\ \ .$$
We note that
$$<\Psi(fa), x>=Tr(fax)=Tr(aS^{D}(f)x)=<\Psi(a), S^{D}(f)x>=<f*\Psi(a), x> $$
for all $  a , x \in A_{K}, f\in A_{K}^{D}\ .$ Therefore $\Psi$ is an isomorphism of $A_{K}^{D}$-modules. It follows  from the
definition of the codifferent that $\mathcal{D}^{-1}(A)=\Psi^{-1}(A^D)$. We now  consider $\Psi (A)$.
Since $\Psi$ is an isomorphism of $A_{K}^{D}$-modules and since  $A=A^{D}I$ we obtain that
$\Psi (A)=\Psi (A^{D}\Lambda e )=\Lambda A^{D}\Psi (e )$. Therefore we are
reduced  to determining $\Psi (e )$.  Let $x$ be an element of $A_{K}$. From the direct sum  decomposition of $A_{K}$, it follows that $x$ can be written as a sum $\lambda e +x'$ with $x'\in \mbox{Ker} \varepsilon $ and $\lambda \in K$. Hence we have
$$<\Psi (e ), x>=Tr(e x)=Tr( e\lambda +e x ')\ \ .$$ We note that $e x'=0$.  Moreover, since
$e $ is a non-trivial idempotent whose $K$-span has dimension one, its trace is $1$. Therefore we have proved that $<\Psi (e ), x>=\lambda =\varepsilon (x)$
for all $x\in A_{K}$. We conclude that $\Psi(e)=\varepsilon $ which is the unit element  of $A^{D}_{K}$. Therefore we have
proved   that
$\Psi (A)=\Lambda A^{D}$ and thus  $\mathcal{D}^{-1}(A)=\Lambda ^{-1}A$ as required.
\end{proof}
\medskip
\vskip 0.1 truecm
\noindent {\bf Remark} Observe that $\Lambda$ is the $R$-ideal defined by
$$\Lambda=\varepsilon (I)\ . $$

\begin{cor} Assume that $I(A)$ is a free $R$-module. Then:
\begin{itemize}
\item[i)] $I(A^{D})=\Lambda^{-1}t$, where $t$ is the unique element of $A^{D}_{K}$ such that $te=1_{A_{K}}$.

\item[ii)] $Tr(x)=tx$ for any $x\in B_{K}$.

\item[iii)] $\mathcal{D}^{-1}(B)=\Lambda^{-1}B$.
\end{itemize}

\end{cor}

\begin{proof} Let $\lambda \in \Lambda$ be such that $\theta=\lambda e$ is a basis of $I(A)$. We note that
$\lambda^{-1}t$ is the unique element $\theta^{D} \in A^{D}_{K}$ such that $\theta^{D}\theta= 1$. Since there exists
such an element in $A^{D}$, we may conclude that $\theta^{D} \in A^{D}$. It  follows from Lemma 2.2 that $\theta^{D}$ is an $R$-basis
of $I(A^{D})$ and thus $i)$ is proved. We now fix an extension $K'/K$ as in Lemma 2.4 and we identify on the one hand the algebras
$A_{K'}$ and $\mbox {Map}(\Gamma, K')$ and on the other hand the algebras $A^{D}_{K'}$ and $K'[\Gamma]$. Since $A_{K}$ is contained in
$A_{K'}$ we observe that $e$ is the unique idempotent in $\mbox {Map}(\Gamma, K')$ such that $I_{K'}=K'e$. Therefore, as an element of
$\mbox {Map}(\Gamma, K')$, $e$ is defined by $e(\gamma)=1$ if $\gamma =1$ and $0$ otherwise. Let
$\omega_{\Gamma}=\sum_{\gamma\in \Gamma }\gamma$. One can easily check that $\omega_{\Gamma}$ is the unique element in
$K'[\Gamma]$ such that $\omega_{\Gamma}e=1_{A_{K'}}$. Therefore we deduce that $t=\omega_{\Gamma }$. We now have
$$Tr_{B_{K}/K}(x)=Tr_{B_{K'}/K'}(x)=\omega_{\Gamma }x=tx, \ \ \forall x \in B_{K}\ $$ as required. Let $\lambda$ be
a generator of $\Lambda$. It follows from $ii)$ that $Tr(\lambda^{-1}x)=(\lambda^{-1}t)x$, for any $x\in B$. Since we know from $i)$
that $\lambda^{-1}t \in A^{D}$ we deduce that $Tr(\lambda^{-1}x)\in B\cap K=R$ for any $x\in B$ and thus that $\Lambda^{-1}B$
is contained in $\mathcal{D}^{-1}(B)$. In order to prove the equality, since $B$ is faithfully flat over $R$, it suffices to prove
that the injection $\Lambda^{-1}B\rightarrow \mathcal{D}^{-1}(B)$ induces an isomorphism
$B\otimes_{R}\Lambda^{-1}B\simeq B\otimes_{R}\mathcal{D}^{-1}(B)$. We first observe that $B\otimes_{R}B$ (resp.
$B\otimes_{R}A$) is a subalgebra of $B_{K}\otimes_{K}B_{K}$ (resp. $B_{K}\otimes_{K}A_{K}$). Moreover since
$B_{K}\otimes_{K}B_{K}$ and $B_{K}\otimes_{K}A_{K}$ are both free and finite $B_{K}$-algebras we can consider  the traces over
$B_{K}$ of these algebras and therefore the  codifferents of $B\otimes_{R}B$ and $B\otimes_{R}A$ for these  traces  forms.
 Let $\varphi=(I_{d}\otimes 1, \alpha_{B }): B\otimes_{R}B\rightarrow B\otimes_{R} A$ be the isomorphism of $B$-algebras  introduced before. It follows from the definitions that
$\mathcal{D}^{-1}(B\otimes_{R}B)=\varphi^{-1}(\mathcal{D}^{-1}(B\otimes_{R}A))$. Moreover one easily checks that
$B\otimes_{R}\mathcal{D}^{-1}(B)$ is contained in $\mathcal{D}^{-1}(B\otimes_{R}B)$. Therefore we conclude that
in order to prove that $\Lambda^{-1}B\rightarrow \mathcal{D}^{-1}(B)$ is surjective it suffices to prove that
$\mathcal{D}^{-1}(B\otimes_{R}A)=B\otimes_{R}\Lambda^{-1}A$. Following the proof of Proposition 3.2 we note that the
previous equality will be a consequence of:

\begin{lem}
$$I(B\otimes_{R}A)=B\otimes_{R}I(A)=B\otimes_{R}\Lambda e\ .$$
\end{lem}
\begin{proof}It follows from the decomposition of $A_{K}$ as a direct sum of ideals that
$$B_{K}\otimes_{K}A_{K}=B_{K}e\oplus (B_{K}\otimes_{K}{\mbox Ker}(\varepsilon))  . $$
Moreover it is immediate that $B_{K}e \subset I(B_{K}\otimes_{K}A_{K})$. Any  element $x$ of $I(B_{K}\otimes_{K}A_{K})$
can be written  $x=\lambda e + y$ with $y\in B_{K}\otimes_{K}{\mbox Ker}(\varepsilon)$. This implies that
$y\in I(B_{K}\otimes_{K}A_{K})$ and therefore that    $ey=\varepsilon (e)y=y=0$. We conclude that
$B_{K}e = I(B_{K}\otimes_{K}A_{K})$. Let $\theta$ be an $R$-basis of $I(A)$.  Since $\theta =ue$,  where $u$ is a non-zero element of  $K$, it follows that $\theta $ is a  $B_{K}$-basis  of $I(B_{K}\otimes_{K}A_{K})$. Let $z$ be an element of $I(B\otimes_{R}A)$.
 There exists $\alpha \in B_{K}$ such that $z=\alpha \theta $. This implies that $\theta^{D}z=\alpha$ and thus that $\alpha\in
 B\otimes_{R}A$. Therefore we have proved that $I(B\otimes_{R}A)$ is contained in $B\otimes_{R}I(A)$. Since the other inclusion
 is obvious we have the equalities. This achieves the proof of the lemma and of the corollary.
 \end{proof}
 \end{proof}

 \begin{cor} We assume that $A_{K}$ is separable of rank $n$ and that $I(A)$ is $R$-free.
Let $\theta $ (resp $\theta^{D}$) denote an integral of $A$ (resp. $A^{D}$) such that $\theta^{D}\theta =1$. Then
$\varepsilon (\theta)\varepsilon^{D}(\theta^{D})=n $. In particular if $\Lambda=\varepsilon (I) $ and
$\Lambda ^{D}=\varepsilon ^{D}(I^{D})$ then $\Lambda \Lambda ^{D}= nR$.
\end{cor}

\begin{proof} Let $Tr$ be the trace form on $A_{K}$. First observe that $Tr(\theta^{D}\theta)=n$. Moreover,  since the trace is
equivariant, it follows that
$$n=Tr(\theta^{D}\theta)=Tr(\theta. S^{D}(\theta^{D})1_{A})=\varepsilon^{D}(\theta^{D})Tr(\theta)\ .$$
Under our hypothesis,  it follows from Proposition 3.2  that there exists $\lambda \in  R$ such that
$$\Lambda=\lambda R\ \ \mbox{and}\  \ \theta=\lambda e\ .$$
It follows from the direct sum decomposition of $A_{K}$ that $Tr(e)=1$. Since $\varepsilon (e)=1$ we deduce  from the previous
equality that $Tr(\theta)=\lambda=\varepsilon (\theta)$ and so that $\varepsilon^{D}(\theta^{D})\varepsilon (\theta)=n $
and  $\Lambda^{D}\Lambda=nR$
\end{proof}
\medskip

\vskip 0.1 truecm
\noindent {\bf Remark.} It can be proved that   $\mathcal{D}(A)$ is the Fitting ideal of the module of differentials
$\Omega ^{1}_{A/R}$. It therefore follows from the lemma
that,  if  $n$ is a unit of $R$, then the module  $\Omega ^{1}_{A/R}$ is trivial. In this case the cover of schemes
$(\mbox{Spec}(B)\rightarrow \mbox{Spec}(R))$ is \'etale for any principal homogeneous space $B$.

 \subsection {Twists of a form by a principal homogeneous space}

 In this section we need a strengthening of hypothesis ${\bf H_{1}}$. The fact that this new hypothesis implies $\bf H_{1}$ follows from Lemma 2.4.
 \medskip

 \noindent {\bf Definition 2} {\it A finite and flat $R$-Hopf algebra $A$ satisfies hypothesis ${\bf H_{2}}$ when
   $A_{K}$ is a commutative separable $K$-algebra and the image under $\varepsilon $ of the set of integrals of $A$ is
   the square of a principal ideal of $R$. }
   \medskip

 Let $A$ be a Hopf algebra satisfying ${\bf H_{2}}$. We denote by $\Lambda^{1/2}$  a principal  ideal of $R$ such that
 $$(\Lambda^{1/2})^{2}=\Lambda=\varepsilon (I(A))\ . $$
 Then, for any principal homogeneous space $B$ of $A$, it follows from Corollary 3.3  that
 $$\mathcal{D}^{-1/2}(B)=\Lambda^{-1/2}B\  $$
 is a square root  of $\mathcal{D}^{-1}(B)$ and that $(\mathcal{D}^{-1/2}(B), Tr)$ is a projective and $A$-equivariant  symmetric bundle on
 $R$. Let us denote by $\lambda^{1/2} $  a generator of $\Lambda^{1/2}$, let $\theta $ be the generator $\lambda e$ of $I(A)$ and let
 $\theta^{D}$ be the unique element of $A^{D}$ such that $\theta^{D}\theta= 1_{A}$. Then, following the
 construction  of  Section 2.2, for any $A$-equivariant symmetric bundle
 $(M, q)$ and any principal homogeneous space  $B$ of $A$,  we can define    the {\it  twist of $(M, q) $ by $B$} (associated to $\theta^{D}$).
 \medskip

 \noindent {\bf Definition 3} {\it Let $A$ be a Hopf algebra satisfying ${\bf H_{2}}$,  let $(M, q)$ be an $A$-equivariant symmetric bundle
 and let $B$ be a principal homogeneous space of $A$. Define the algebraic twist of $(M, q)$ by $B$ as the $R$-symmetric bundle
 $$(\tilde M_{B}, \tilde q_{B})=(\mathcal{D}^{-1/2}(B)\otimes_{R} M, Tr\otimes q)^{A}\ .$$}

 \medskip

 \noindent {\bf Remarks }

 \noindent {\bf 1.}We observe that, since $\theta$ is defined up to the square of a unit of $R$ the same holds for $\theta^{D}$.
 Therefore, as observed in Remark  of Section 2, the definition of   $(\tilde M_{B}, \tilde q_{B})$
 is independent, up to isometry,  of the choice of $\theta$.

 \noindent {\bf 2.} It follows from Proposition 2.8 that under hypothesis ${\bf H_{2}}$ we can attach to $(M, q)$ an
 orthogonal representation $\rho: G=\mbox {Spec}(A)\rightarrow {\bf O}(q)$. With a view to generalizing this definition later on,
 we shall often refer to the twist of $(M, q)$ by $B$ as  {\it the twist of $(M, q )$ by $\rho$ and $X=\mbox{Spec}(B)$}.  We will denote this twist as  $(M_{\rho, X}, q_{\rho, X})$.

 \noindent {\bf 3.} Suppose that  $R=K$ is a field and that $L/K$ is a Galois extension with Galois group $\Gamma$. Let
 $(M, q)$ be the underlying symmetric bundle of an orthogonal representation $\rho: \Gamma \rightarrow {\bf O}(q)$.
 This is a situation where we can apply our previous construction. Let
  $A= \mbox {Map}(\Gamma, K)$ be the Hopf algebra defining the constant group scheme associated to $\Gamma$. Since
 $L$ is a principal homogeneous space for $A$  we can consider the twist of $(M, q)$ by $L$
 $$(\tilde M_{L}, \tilde q_{L})=(L\otimes_{K} M, Tr\otimes q)^{A}\ $$
 as introduced in Definition 3. It follows from [F], Theorem 1 and [CNET], Proposition 2.5,  that this new quadratic form  coincides with
 the one introduced by Fr\"ohlich in [F], Section 2.

 \section {Twists of a form and flat cohomology}

 Let $S$ be the scheme $\mbox{Spec}(R)$ and let $G$ be the $S$-group scheme defined by a Hopf algebra $A$. To any $R$-linear map
$$\alpha_{B}: B\rightarrow B\otimes_{R}A ,  $$
which endows $B$ with the structure of a comodule algebra over $A$, there corresponds a morphism of $S$-schemes
$$X\times_{S} G \rightarrow X \ . $$
In this correspondence the notion of $PHS$ corresponds to the notion of torsors for $G$ over $S$.

Following Milne [M], III, Section 4, we may associate to any
flat covering $\mathcal{U}=(\mathcal{U}_{i}\rightarrow S)_{i\in I}$ a set of cohomology classes $\check H^{1}(\mathcal{U}, G)$;
this is a set with a distinguished element. We define $\check H^{1}(S, G)$ to be the  direct limit over all coverings $\mathcal{U}$ of
$\check H^{1}(\mathcal{U}, G)$. From Theorem 4.3 and Proposition 4.6 in [M], it follows that there exists a one to one correspondence,
$[X]\rightarrow c(X)$,
between the isomorphism classes of $G$-torsors over $S$, that we denote by $H^{1}(S, G)$,  and elements of $\check H^{1}(S, G)$ under which the class of the trivial torsor (the class of $A$)
 corresponds to the distinguished element of $\check H^{1}(S, G)$.

 Let $(M, q)$ denote an $A$-equivariant symmetric bundle and assume that $A$ satisfies ${\bf H_{2}}$. As per Proposition 2.8 we can
 associate to $(M, q)$ a morphism of group schemes $\rho: G\rightarrow {\bf O}(q)$. It is routine to check that
 $\rho$ transforms $1$-cocycles on $G$ into
  $1$-cocycles on ${\bf O}(q)$  and thereby induces a map $\rho_{*}$ from
  $\check H^{1}(S, G)$ in $\check H^{1}(S, {\bf O}(q))$. The set $\check H^{1}(S, {\bf O}(q))$
  classifies the set of isomorphism classes of twisted forms of $(M, q)$ ([D-G], III, Section 5, n.2). Therefore  the class $\rho_{*}(c(X))$ defines, up to isometry,
  a unique symmetric bundle that we denote by  $(M_{\rho (X)}, q_{\rho(X)})$.  We now have at our disposal on the one hand
  the symmetric bundle $(M_{\rho (X)}, q_{\rho(X)}))$, which has an abstract definition in terms of class of a cocycle in a flat cohomology
  set, and on the other hand the algebraic twist  $(M_{\rho, X}, q_{\rho, X})$ given by a simple explicit formula
  (see Definition 3 in Section 3.2). The main goal of this section is to prove that the two bundles coincide.

   \begin {thm} There exists an isometry of  symmetric bundles
  $$(M_{\rho, X}, q_{\rho, X})\simeq (M_{\rho (X)}, q_{\rho(X)})\ . $$
  \end {thm}
  \vskip 0.1 truecm

 We keep the notation and the hypotheses of Section 3. We assume  that $A$ satisfies ${\bf H_{2}}$, and
 in particular we assume that the image under $\varepsilon $ of the set of integrals of $A$ is the square of a principal ideal of $R$. We fix a
 generator $\lambda^{1/2}$ of this ideal. Since $A_{K}$ is separable there exists  a
 finite extension $K'/K$ and a finite group $\Gamma $ such that $A_{K'}=\mbox{Map}(\Gamma, K')$ and $A_{K'}^{D}=K'[\Gamma]$.
 For the sake of notational simplicity we shall assume  that $K'=K$; the general case can  be  easily deduced. We let $e$ be the element of
 $A_{K}$ defined by $e(\gamma)=1$ if $\gamma=1$ and $0$ otherwise and we denote by  $\omega $ the element
 $\sum_{\gamma\in \Gamma}\gamma$ in $K[\Gamma]$. We have seen in section 3.2  that $\theta=\lambda e$ (resp.
 $\theta^{D}=\lambda^{-1}\omega)$ is a $R$-basis of $I(A)$ (resp. $I(A^{D})$) and that we have the equalities

 $$\theta ^{D}\theta=1_{A},\ \theta \theta^{D}=1_{A^{D}},\ A=A^{D}\theta,\ A^{D}=A\theta^{D}\ . $$

 \subsection {Representative of a torsor}

 Let $B$ be a PHS of $A$ and let $X=\mbox{Spec}(B)$ be the associated $G$-torsor.  It follows from the definition that there is a flat cover
 $\mathcal{U}=(X\rightarrow S)$ which trivializes $X$. More precisely
 the isomorphism
 $$\varphi=(i_{d}\otimes 1, \alpha_{B}):B\otimes_{R}B\rightarrow B\otimes_{R}A , $$  induces an isomorphism of
 $S$-schemes with $G$-action
  $$\Phi=\mbox{Spec}(\varphi): X\times_{S}G\rightarrow X\times_{X}X\ .$$
  Let $p_{1}$ (resp. $p_{2}$) denote the  first (resp. second) projection map $X\times_{S}X \rightarrow X$. For
  $1\leq i\leq 2$ the base change of $\Phi$ by $p_{i}$ defines an isomorphism of schemes with $G$-action
  $$\Phi_{i}: (X\times_{S}X)\times_{S}G \rightarrow (X\times_{S}X)\times_{S}X \ .$$
  We know from p. 134 in[M] that  $\Phi_{1}^{-1}\circ \Phi_{2}$ is a $1$-cocycle representing $c(X)$. We wish to understand
  $\Phi_{1}^{-1}\circ \Phi_{2}$ in terms of $B\otimes_{R}B$-valued points of $G$.

   Let $q_{1}$ (resp. $q_{2}$) denote the morphism
  of algebras $ B\rightarrow B\otimes_{R}B$, defined by $(q_{1}: x\rightarrow x\otimes1)$
  (resp. $q_{2}:x\rightarrow  1\otimes x $). Extending scalars by $q_{i}$ for $1\leq i\leq 2$, the map $\varphi $ induces an
  isomorphism
  $$\varphi_{i}: (B\otimes_{R}B)\otimes_{R}B\rightarrow (B\otimes_{R}B)\otimes_{R} A$$
  of $B\otimes_{R}B$-algebras  and $A^{D}$-modules. It is clear that
  $\Phi^{-1}_{1}\circ \Phi_{2}=\mbox{Spec}(\varphi_{2}\circ \varphi_{1}^{-1})$.

  Let  $C $ be  the algebra $B\otimes_{R}B$. We recall the identifications of Section 2.2:
  $$G(C)=\mbox{Hom}_{alg, R}(A, C)=\mbox{Hom}_{alg, C}(A_{C}, C)\ . $$
  We denote by
  $\underline {Aut}(A_{C})$ the group of  automorphisms of $C$-algebra and $A_{C}$-comodule.
  We  observe that
  $\varphi=\varphi_{2}\circ \varphi_{1}^{-1}$ is an element of this group. For any element $\psi$ of  $\underline {Aut}(A_{C})$
  and $f\ \in G(C)$ we obtain a element of $G(C)$ by considering $f\circ \psi$.

  \begin {lem} The map
  $$\theta: \underline {Aut}(A_{C})\rightarrow G(C)$$
  $$\psi\rightarrow \varepsilon\circ\psi$$
  is a group isomorphism.
  \end{lem}
  \begin{proof} Since any $\psi \in \underline {Aut}(A_{C})$ is a morphism of comodules it satisfies the equality
  $$(\psi\otimes id)\circ\Delta=\Delta\circ \psi.$$
  This implies that  for all  $x\in A_{C}$,
  $$\Delta(\psi(x))=\sum_{(x)}\psi(x_{0})\otimes x_{1} $$
  where $\Delta(x)=\sum_{(x)} x_{0}\otimes x_{1}$ and so $\psi(x)=\sum_{(x)} (\varepsilon\circ \psi)(x_{0})x_{1}$.
  For any $\psi_{1}$ and $\psi_{2}$ of $\underline {Aut}(A_{C})$  and $x\in A_{C}$ we obtain that
  $$\theta(\psi_{1})\theta(\psi_{2})(x)=\sum_{(x)}(\varepsilon\circ\psi_{1})(x_{0})(\varepsilon \circ \psi_{2})(x_{1})
  =(\varepsilon \circ \psi_{2})(\sum_{(x)}(\varepsilon\circ\psi_{1})(x_{0})x_{1}).  $$
 This shows that $$\theta(\psi_{1})\theta(\psi_{2})(x)=(\varepsilon \circ \psi_{2}\circ \psi_{1})(x)=
 \theta((\psi_{1})\circ(\psi_{2}))(x) .$$We conclude that $\theta$ is a group homomorphism. Suppose now that
 $\theta(\psi)=\varepsilon$. Then it follows from the previous equalities that for all $x\in A_{C}$
 $$\psi(x)=\sum_{(x)}\varepsilon (x_{0})x_{1}=x, . $$
 Therefore $\theta $ is an injection. Let us now consider an element $\alpha\in G(C)$ and denote by $\psi$ the $C$-endomorphism of
 $A_{C}$ defined by
 $$\psi(x)=\sum_{(x)}\alpha(x_{0})x_{1}.$$
 Since $\alpha$ is a morphism of algebras it follows that $\psi $ is a morphism of $C$-algebras on $A_{C}$. Moreover, using the equality
 of morphisms
 $$(id\otimes\Delta)\circ\Delta=(\Delta\otimes id)\Delta$$ one can check that $\psi$ is a morphism of $A_{C}$-comodules. Finally
 we observe that for all $x\in A_{C}$
 $$\theta(\psi)(x)=\sum_{(x)}\alpha(x_{0})\varepsilon(x_{1})=\alpha(\sum_{(x)}x_{0}\varepsilon(x_{1}))=\alpha(x).$$ Therefore $\theta(\psi)=\alpha$. This proves that $\theta $ is onto and ends the proof of the lemma.
 \end{proof}
\medskip

  We deduce from this lemma  that  the map $\Psi= \mbox {Spec}(\psi)\rightarrow \varepsilon \circ \psi$ is
  an isomorphism of groups from $\mbox {Aut}(\mbox {Spec}(C)\otimes_{S}G)$ onto $G(C)$. We identify these groups via this isomorphism.
  Under this identification,  we conclude that $g=\varepsilon\circ \varphi$, where
  $\varphi=\varphi_{2}\circ \varphi_{1}^{-1}$ is the element of $\underline {Aut}(A_{C}) $ introduced previously, is a $1$-cocycle
  representative of $c(X)$ in $G(C)$.
  \vskip 0.1 truecm

  \noindent {\bf Remark } As we have seen previously,  for any $x\in A_{C}$, with $\Delta (x)=\sum_{(x)} x_{(0)}\otimes x_{(1)}$
  we  have the
  equality:
  $$\varphi(x)=\sum (\varepsilon \circ \varphi) (x_{(0)}) x_{(1)}\ .$$
  In the  case where $\Delta (x)$ is invariant under the twist map (which is the map induced by $c\otimes d\rightarrow d\otimes c$), this
  last equality can be written
  $$\varphi(x)=(\varepsilon \circ \varphi)x , $$
  where $A_{C}$ is endowed with its structure of left $A^{D}_{C}$-module and where  $\varepsilon \circ \varphi$ is  considered as
  an element of $A^{D}_{C}$ via  the inclusion $G(C)\subset A^{D}_{C}$. When $A$ is cocommutative, since every element of
  $A\otimes A$ is invariant under the twist map, the automorphism $\varphi$ is the map $x\mapsto (\varepsilon \circ \varphi)x$.

  \subsection {Proof of Theorem 4.1}

  Let $(M, q)$ be an $A$-equivariant symmetric bundle and let $B$ be a PHS of $A$. We consider the twist
    of $(M, q)$ by $B$ defined by

    $$(M_{\rho, X}, q_{\rho, X})=(\tilde M_{B}, \tilde q_{B})=(\lambda^{-1/2}B\otimes_{R} M, Tr\otimes q)^{A}\ . $$

    The strategy for the proof of the theorem is to show that the flat covering $\mathcal{U}=(X\rightarrow S)$,  which
  trivializes $X$ as a $G$-torsor,  likewise trivializes the symmetric bundle $(M_{\rho, X}, q_{\rho, X})$. With this in view  we
  now construct an isometry of symmetric bundles
  \begin{equation}
  (M_{B}, q_{B})\simeq (B\otimes_{R}\tilde M_{B},  \tilde q_{B, B}),
  \end{equation}
  where $( \tilde q_{B, B})$ denotes the form $\tilde q_{B}$ extended to $B\otimes_{R}\tilde M_{B}$.
  This construction will be achieved via the next two lemmas. The results of  Section 4.1 provide us with  a representative of $X$ in
  $\check H^{1}(\mathcal{U}, G)$. We will use the  previous isometry  to show that the image under $\rho_{*}$
  of this cocycle is a representative of the class of $(M_{\rho, X}, q_{\rho, X})$.

 \begin {lem} Let $T$ be a finite flat $R$-algebra. Then

 \begin{itemize}
\item[i)]  For any $f\in A^{D}_{T}$ and $m\in M_{T}$ one has
$$\theta^{D}(f\theta \otimes m)=\theta^{D}(\theta\otimes S^{D}(f)m)\ . $$

\item [ii)] The map $m\mapsto  \theta^{D}(\lambda^{-1/2}\theta \otimes m)$ induces an isometry $\nu_{T}$ of symmetric
bundles from $(T\otimes_{R}M, q_{T})$ onto $(T\otimes_{R}\tilde M_{A}, \tilde q_{A, T})$ where
$(\tilde M_{A}, \tilde q_{A})$ is the symmetric bundle $(\lambda ^{-1/2}A\otimes_{R} M, Tr\otimes q)^{A}$.

\end{itemize}

\end{lem}

\begin{proof} We first observe that it suffices to prove the lemma when $T=R$. The general case will follow by extension of scalars. Let
$f$ be an element of $A^{D}$. Since $A^{D}$ is contained in $A^{D}_{K}=K[\Gamma ]$ we write $
f=\sum _{\gamma \in \Gamma }x_{\gamma}\gamma$ with $x_{\gamma}\in K$. Since $A^{D}$ acts diagonally over
$A\otimes_{R}M$ and since $\theta^{D}\gamma=\theta^{D}$,   we obtain  that for any
$m \in M $  and  $\gamma \in \Gamma $
$$\theta^{D}(f\theta \otimes m)=\sum_{\gamma \in \Gamma}\theta^{D}(x_{\gamma}\gamma \theta\otimes m)=
\sum_{\gamma \in \Gamma}\theta^{D}\gamma( \theta\otimes x_{\gamma}\gamma^{-1}m)
=\theta^{D}(\theta\otimes S^{D}(f)m), $$
as required.

Let  $\nu: M\rightarrow \tilde M_{A}$ be the $R$ linear map defined by $m\mapsto \theta^{D}(\lambda^{-1/2}\theta\otimes m)$.
We start by  proving  that $\nu$ is surjective. Let $y$ be an element of $M^{A}$. We deduce  from Lemma 2.5 the existence of  a finite set
of elements $x_{i}$ of $A$ and $m_{i}$ of $M$ such that
$y=\sum_{i}\theta^{D}(\lambda^{-1/2}x_{i}\otimes m_{i})$. Since $\theta $ is an $A^{D}$-basis of $A$,
for an integer $i$,  there exists an
element $f_{i}$ of $A^{D}$ such that $x_{i}=f_{i}\theta $.
Using  i) we deduce that
$$y=\sum_{i}\theta^{D}(f_{i}\lambda^{-1/2}\theta \otimes m_{i})=
\sum_{i}\theta^{D}(\lambda^{-1/2}\theta\otimes S^{D}(f_{i})m_{i})=\theta^{D}(\lambda^{-1/2}\theta\otimes x), $$
with $x=\sum_{i}S^{D}(f_{i})m_{i}$.  So we have found  $x \in M$  such that  $y=\nu (x)$.

Let us now consider
$\tilde \varepsilon: A_{K}\otimes_{K}M_{K}\rightarrow M_{K}$ defined by $a\otimes m\mapsto <\varepsilon, a>m$. Since the
action of $A^{D}_{K}$ is diagonal, for any $m\in A_{K}$ we have
$$\tilde \varepsilon (\theta^{D}(\theta\otimes m))=(\alpha)m\ , $$
where  $\Delta^{D}(\theta^{D})=\sum \theta^{D}_{(0)}\otimes \theta^{D}_{(1)}$
and $\alpha=\sum <\varepsilon, \theta^{D}_{(0)}\theta)>\theta^{D}_{(1)}$.  We first observe that
$<\varepsilon, \theta^{D}_{(0)}\theta)>=<\theta^{D}_{(0)}, \theta >$. Moreover, since $A$ is commutative we know that
$A^{D}$ is cocommutative. It follows from these facts that
$$\alpha=\sum <\theta_{(0)}, \theta>\theta^{D}_{(1)}=
\sum <\theta^{D}_{(1)}, \theta>\theta^{D}_{(0)}=\theta \theta^{D}\ . $$ Since,  as recalled at the beginning of this section,
we know that $\theta \theta^{D}=1_{A^{D}}$, then  we have proved that
$$\tilde \varepsilon (\theta^{D}(\theta\otimes m))=m, \ \forall m \in M_{K}\ .$$ We conclude that
$\nu$ is an isomorphism whose inverse is given by $\mu: x\mapsto \lambda^{1/2}\tilde \varepsilon (x)$. In order to complete
the proof of the lemma we must show that $\nu$ is an isometry. Hence we have to check that
$$\tilde q_{A}(\theta^{D}(\lambda^{-1/2}\theta\otimes m), \theta^{D}(\lambda^{-1/2}\theta\otimes m)')=
q(m, m'), \ \forall m, m' \in M\ . $$
The proof now consists of a  verification of this equality. It follows from the definitions that
$$\tilde q_{A}(\theta^{D}(\lambda^{-1/2}\theta\otimes m), \theta^{D}(\lambda^{-1/2}\theta\otimes m'))=
\lambda^{-1}(Tr\otimes q)(\theta^{D}(\theta\otimes m), \theta\otimes m')\ .$$
Since the action of $A^{D}$ on $A\otimes_{R}M $ is diagonal we can write
$\theta^{D}(\theta\otimes m)$ as the sum $\sum \theta^{D}_{(0)}\theta \otimes \theta^{D}_{(1)}m$ and so
$$(Tr\otimes q)(\theta^{D}(\theta\otimes m), \theta\otimes m')=
\sum Tr((\theta^{D}_{(0)}\theta)\theta)q(\theta^{D}_{(1)}m, m')\ . $$
It follows from the previous equalities that
$$\tilde q_{A}(\theta^{D}(\lambda^{-1/2}\theta\otimes m), \theta^{D}(\lambda^{-1/2}\theta\otimes m'))=
\lambda^{-1}q(\beta m, m'), \ \mbox{with }\ \beta= \sum (Tr(\theta^{D}_{(0)}\theta)\theta)\theta^{D}_{(1)}\ .$$
  Our goal is to compute $\beta$.
We start by evaluating $(\theta^{D}_{(0)}\theta)\theta$.   Writing  $\Delta (\theta)=\sum \theta_{(0)}\otimes \theta_{(1)}$
we obtain  that $\theta^{D}_{(0)}\theta=\sum \theta_{(0)}<\theta^{D}_{(0)}, \theta_{(1)}>$. Hence
$$(\theta^{D}_{(0)}\theta)\theta= \sum \theta_{(0)}\theta <\theta^{D}_{(0)}, \theta_{(1)}>
=\sum \varepsilon(\theta_{(0)})\theta <\theta^{D}_{(0)}, \theta_{(1)}>=
\theta <\theta^{D}_{(0)}, \sum \varepsilon (\theta_{(0)})\theta_{(1)}> .$$
These equalities show that $(\theta^{D}_{(0)}\theta)\theta= <\theta^{D}_{(0)}, \theta>\theta$ and   therefore that

$$Tr((\theta^{D}_{(0)}\theta)\theta)=<\theta^{D}_{(0)}, \theta>Tr(\theta)\ . $$ It now follows from  Corollary 3.3 that
$Tr(\theta)=\lambda te= \lambda$.  We conclude  that
$$\beta= \lambda\sum <\theta^{D}_{(0)}, \theta>\theta^{D}_{(1)}=\lambda \ , $$
since,  as observed before,  $\sum <\theta^{D}_{(0)}, \theta>\theta^{D}_{(1)}=\theta \theta^{D}=1_{A^{D}}$.
Hence  $\lambda^{-1}q(\beta m, m')=q(m, m')$,  as required.
\end{proof}

\vskip 0.2 truecm

\noindent {\bf Remark } As a  consequence of the lemma we observe  that there is an isometry
          $$(M, q)\simeq (\tilde M_{A}, \tilde q_{A})$$
which proves  Theorem 4.1  in the case when $X$ is the trivial torsor.

We now return to  the isomorphism $\varphi: B\otimes _{R}B\rightarrow  B\otimes _{R}A $ introduced at the beginning of this section.
This morphism induces an isomorphism
$$\tilde \varphi=(\varphi\otimes Id): B\otimes _{R}(\lambda^{-1/2}B\otimes_{R}M)\rightarrow
 B\otimes _{R}(\lambda^{-1/2}A\otimes_{R}M) $$
 of $B$ and $A^{D}$-modules. Recall that $A^D$ acts on the left-hand side  (resp. the right-hand side) via its diagonal action
 on $(\lambda^{-1/2}B\otimes_{R}M)$ (resp. $(\lambda^{-1/2}A\otimes_{R}M)$). Therefore,  taking fixed points by $A^{D}$,
 we see that $\tilde \varphi$ induces an isomorphism of $B$-modules
 $$\tilde \varphi: B\otimes_{R}\tilde M_{B}\rightarrow B\otimes_{R}\tilde M_{A}\ . $$

 \begin {lem} The isomorphism $\tilde \varphi$ induces an isometry of symmetric bundles
 $$(B\otimes_{R}\tilde M_{B}, \tilde q_{B, B})\simeq  (B\otimes_{R}\tilde M_{A}, \tilde q_{A, B}). $$
 \end {lem}
 \begin{proof}This is an easy verification that we leave to the reader.
 \end{proof}

 We  can now complete the proof of the theorem. It follows from Lemmas 4.3 and 4.4 that
 $$\tilde \varphi^{-1}\circ \nu_{B}: (B\otimes_{R}M, q_{B})\rightarrow (B\otimes_{R }\tilde M_{B}, \tilde q_{B, B})$$
 is an isometry.
 Let $C=B\otimes_{R}B$ and let $q_{i}: B\rightarrow C, 1\leq i\leq 2$
 be the  morphisms as considered prior to Lemma 4.2 at the beginning of this section. For $1\leq i \leq 2$
  the map $\tilde \varphi^{-1}\circ \nu_{B}$ induces, by scalar extension,  an isometry:
 $$\tilde \varphi^{-1} _{i}\circ \nu_{C}:
 (C\otimes_{R}M,  q_{C})\rightarrow  (C\otimes_{R} \tilde M_{B}, \tilde q_{B, C}) \ .$$
 This implies that $\sigma=(\tilde \varphi^{-1} _{1}\circ \nu_{C})^{-1}\circ (\tilde \varphi^{-1} _{2}\circ \nu_{C})
 =\nu_{C}^{-1}\circ (\tilde \varphi_{1}\circ \tilde \varphi^{-1}_{2})\tilde \nu_{C}$ is a $1$-cocycle representative of
 $(M_{\rho, X}, q_{\rho, X})$. Our goal is now to describe this map.

 Let $c$ (resp. $m$) be an element of $C$ (resp. $M$).
  Since $\tilde \varphi_{1}\circ \tilde  \varphi^{-1}_{2}$ commutes with action of $A^{D}$ it follows from the definitions that

 $$(\tilde \varphi_{1}\circ \tilde  \varphi^{-1}_{2})\circ \nu_{C}(c\otimes m)=
 \theta^{D}((\varphi_{1}\circ   \varphi^{-1}_{2})(\lambda^{-1/2}\theta)\otimes( c \otimes m))\ . $$
 It is easily checked that  that $\Delta (\theta)=\lambda\Delta ( e) $ is invariant under the twist map. We then deduce from
 Remark, Section 4.2  that $(\varphi_{1}\circ   \varphi^{-1}_{2})(\lambda^{-1/2}\theta)=\lambda^{-1/2 }(g^{-1}\theta)$, where
 $g$ is the representative of $c(X)$ we constructed previously in Subsection 4.1. It now follows from Lemma 4.3 that
 $$ \theta^{D}((\varphi_{1}\circ   \varphi^{-1}_{2})(\lambda^{-1/2}\theta)\otimes( c \otimes m))=
 \theta^{D}(\lambda^{-1/2}(g^{-1}\theta)\otimes (c\otimes m))
 =\theta^{D}(\lambda^{-1/2}\theta\otimes g(c\otimes m))\ .$$
 This implies that $\sigma (c\otimes m)=
 \nu^{-1}_{C}(\theta^{D}(\lambda^{-1/2}\theta\otimes g(c\otimes m)))=g(c\otimes m)$.
 This tells us  that   $\rho (g)$ is a representative of $(M_{\rho, X}, q_{\rho, X})$ and so achieves the proof of the theorem.

 \section {Examples}
 \subsection  {  The unit form}

\noindent  Let $A$ be a commutative,  finite and flat  Hopf algebra over a principal ideal domain $R$.
Let  $\theta  $ denote a generator of the module of  integrals of $A$.
 We define the form    $\kappa$   on $A^{D}\times A^{D}$  by the equality:
$$\kappa (u, v)=<S^{D}(u)v, \theta $$
for all $ u, v \in A^{D}$.
\begin {prop}  The following properties hold.
\begin{itemize}
\item[i)] The pair $(A^{D}, \kappa)$ is an $A$-equivariant symmetric bundle.

\item [ii)]  If  $A$ satisfies ${\bf H_{2}}$, then, for any principal homogeneous space $B$ of $A$,
the twist of $(A^{D}, \kappa)$ by $B$ coincides with the symmetric bundle $(\mathcal{D}^{-1/2}(B), Tr)$.

\end{itemize}

\end {prop}
\begin{proof}f Since we know from Lemma 2.3 that  $S(\theta )=\theta $, we note that for all $u, v \in A^{D}$
$$\kappa (u, v)= <S^{D}(u)v, \theta >=<S^{D}(u)v, S(\theta )>=<S^{D}(v)u, \theta>=\kappa(v, u) $$
 so that the form is indeed symmetric. We now observe that
 $$\kappa(tu, v)=<S^{D}(tu)v, \theta >= <S^{D}(u)S^{D}(t)v, \theta >= \kappa (u, S^{D}(t)v)\ .$$
 We therefore conclude that the form is $A$-equivariant. In order  to show that the form is non-degenerate,
 we consider  an $R$-basis $\{e_{1}, ..., e_{n}\}$  of $A$ and we let  $\{e_{1}^{*},..., e_{n}^{*}\}$ denote  its dual $R$-basis in $A^{D}$. Since
 $\theta $ is a basis of $A$ as an  $A^{D}$-module,  for any $i, 1\leq i\leq n$, there exists $f_{i} \in A^{D}$ such
 that $S(e_{i})=f_{i}\theta $. It is easily verified  that $\{f_{1}, ..., f_{n}\}$ is an $R$-basis of $A^{D}$. For any $i$ and $j$ we have
 $$\delta _{ij}=<e_{i}^{*}, e_{j}>=<S^{D}(e_{i}^{*}), S(e_{j})>=<S^{D}(e_{i}^{*}), f_{j}\theta > $$
 and therefore
 $$\delta _{ij}=\varepsilon (S^{D}(e_{i}^{*})(f_{j}\theta ))=
 \varepsilon ((S^{D}(e_{i}^{*})f_{j})\theta )=<S^{D}(e_{i}^{*})f_{j}, \theta >=\kappa(e_{i}^{*}, f_{j})\ .$$

 Since $\{f_{1}, ..., f_{n}\}$  is an  $R$-basis of $A^{D}$, we
 can  write $e^{*}_{j}= \sum _{1\leq k\leq n}r_{kj}f_{k}$ where $(r_{ij})$ is a $n\times n$ invertible matrix with coefficients in $R$.
The previous equalities show that   $\kappa (e_{i}^{*}, e_{j}^{*})=r_{ij}$ for any $i$ and $j$. Therefore the form
$\kappa$ is non-degenerate.

\vskip 0.1 truecm

In order to prove(ii) we shall now assume that $A$ satisfies ${\bf H_{2}}$.  In this case we can  provide   a new  description of $(A^{D}, \kappa)$.
From now on we use  the notation of Section 3 and Corollary 3.3.
We let  $\theta$ (resp. $\theta^{D}$)  be the  generator $\lambda e$ (resp. $\lambda^{-1}t $)  of
    $I(A)$ (resp. $I^{D}(A)$) ,  where
$\lambda$ is an $R$-basis  of $\varepsilon (I(A))$. We consider  the map $\varphi: u\mapsto \lambda^{-1/2}u\theta$ from $A^{D}$
into $\mathcal{D}^{-1/2}(A)$.  We wish to show
that this  isomorphism of $A^{D}$-modules induces  an isometry from
$(A^{D}, \kappa)$ into $(\mathcal{D}^{-1/2}(A), Tr)$. Hence we need to show that for all $u, v \in A^{D}$
$$\kappa(u, v) =Tr((\lambda^{-1/2}u\theta) (\lambda^{-1/2}v\theta)).$$
It follows from the definitions that
 $$<S^{D}(u)v, \theta>=\lambda<S^{D}(u)v, e>$$
  while
$Tr((\lambda^{-1/2}u\theta) (\lambda^{-1/2}v\theta))=\lambda t((ue)(ve)$.
Writing    $u=\sum_{\gamma \in \Gamma}u_{\gamma}\gamma$ and $v=\sum_{\delta \in \Gamma}v_{\delta}\delta $,  we easily check that

$$<S^{D}(u)v, e>=t((ue)(ve)=\sum_{\gamma\in \Gamma}u_{\gamma}v_{\gamma}, $$
which is the required equality.
\end{proof}

  Let $B$ be a $PHS$ for $A$. We wish to describe the twist of $(A^{D}, \kappa)$ by $B$. This can be easily done:
 from the very definition of the twist and from our previous observations we obtain that
 $$(\tilde A^{D}_{B}, \tilde \kappa_{B})\simeq (\mathcal{D}^{-1/2}(B)\otimes A^{D},Tr\otimes \kappa)^{A}\simeq
 (\mathcal{D}^{-1/2}(B)\otimes_{R} \mathcal{D}^{-1/2}(A),Tr\otimes_{R} Tr)^{A}\ .$$
 We now deduce from Lemma 3.3 ii) that
 $$(\mathcal{D}^{-1/2}(B)\otimes_{R} \mathcal{D}^{-1/2}(A),Tr\otimes_{R} Tr)^{A}\simeq
 (\mathcal{D}^{-1/2}(B), Tr)\ . $$
 This proves  that $ (\mathcal{D}^{-1/2}(B), Tr)$ is the twist of $(A^{D}, \kappa)$ by $B$.
 \vskip 0.1 truecm
 \medskip
\newpage
 \noindent {\bf Remarks}

 \noindent {\bf 1.} The symmetric bundle $((A^{D})^{A}, \kappa^{A})$ is easy to describe.
   It follows from the very definition  that $({A^{D}})^{A}=I(A^{D})$. Choose a basis  $\theta^{D}$  of $I(A^{D})$
   such that $\theta^{D}\theta=1_{A}$.
 The form $\kappa^{A}$  is defined on $I(A^{D})$ by the formula
 $$\kappa^{A}(\theta^{D} x, \theta^{D}  y)=\kappa( x, \theta^{D}y)=xy<\theta^{D}, \theta > \ \forall x, y \in R\ . $$
 Since
 $$<\theta^{D}, \theta > =\varepsilon(\theta^{D} \theta)=1,  $$
 we conclude   that $\kappa^{A}(\theta x, \theta  y)=xy$.

 \noindent {\bf 2.} If  $A$ satisfies ${\bf H_{2}}$, then the integral $\theta$ used in the proof of the proposition has been  chosen
 according to the stipulations of Section 3. It follows  that  $(A^{D}, \kappa)$ is independent of this choice, up to isometry.
 We refer to $(A^{D}, \kappa)$ as the {\it unit form of $A^{D}$}.

 \noindent {\bf 3.} The Hopf algebra
 $A=\mbox{Map}(\Gamma, R)$, with  $\Gamma$  a finite group, obviously satisfies ${\bf H_{2}}$. This is the case where the group scheme
 defined by $A$ is the constant group scheme. In this situation we can choose  for $\theta$ the integral of $A$ defined by  $\theta (g)=1$
 if $g$ is the identity and $0$ otherwise. We observe that $A^{D}$ is the group algebra $R[\Gamma]$ and that the form $\kappa$
  is given on elements of $\Gamma$ by

 $$\kappa (\gamma, \gamma')=<\gamma^{-1}\gamma', \theta > = \delta _{\gamma, \gamma'}\ . $$
 It follows from these equalities  that  $\{\gamma\in \Gamma\}$ is an orthonormal basis for
 $\kappa$ and therefore   that the  symmetric bundle  $(A^{D}, \kappa )$ is the usual {\it unit form of $\Gamma $}.

 \noindent {\bf 4.} When  $G$ is generically constant,  of odd order,  ($A_{K}=\mbox{Map}(\Gamma, K)$,  with $\Gamma$ of odd order),
 we  know from [BL], that after scalar extension to $K$,  the forms become $\Gamma$-isometric. Therefore  we have the following isometries of equivariant symmetric bundles:
 $$(\mathcal{D}^{-1/2}(B), Tr)\otimes_{R}K\simeq (K[\Gamma], \kappa_{K})\simeq
 (A^{D}, \kappa)_{K}\ . $$ This result leads us naturally to compare
 $((\mathcal{D}^{-1/2}(B), Tr))$ and $(A^{D}, \kappa)$ both as $R$-symmetric bundles and also as  $A$-equivariant symmetric bundles
 in the general situation.

 \subsection  { An orthogonal representation of $\mu _{n}$}

  Consider the $R$-algebra  $A=(R[T]/(T^{n}-1))= R[t]$ with the following additional structure: a comultiplication
   $\Delta: A\rightarrow A\otimes_{R}A $ induced by $t\mapsto t\otimes t, $ a counit $\varepsilon $ induced
  by $t\mapsto 1$ and an antipode $S: A\rightarrow A$ induced by $t\mapsto t^{-1}=t^{n-1}$.
This is then a Hopf $R$-algebra which represents  the $S=\mbox{Spec}(R)$-group scheme
  $\mu_{n}$ of  $n$-th roots of unity.   Its dual $A^{D}=\mbox{Hom}_{R}(R[t], R)$
  represents the constant group scheme ${\bf Z}/n{\bf Z}$ over $S$.

We consider   the symmetric bundle $(V, q)$ consisting of the  $R$-free   module $V$,  of  rank $2$,  with basis
$\{\varepsilon _{1}, \varepsilon _{2}\}$ and the  symmetric bilinear form $q$ defined by
$$q(\varepsilon _{k}, \varepsilon _{k})=0\ \mbox{for}\  1\leq k\leq 2\ \ \mbox{and}\  q(\varepsilon_{1}, \varepsilon _{2})=1/2\ .$$
We note  that when $R$ contains a square root of $-1$,   denoted by
$i$, then   $e_{1}=\varepsilon_{1}+ \varepsilon _{2}$ and $e_{2}= i(\varepsilon _{1}-\varepsilon _{2})$ form  a new
 basis of   $V$ for which $q(e_{k},e_{k})=1,\ \mbox{for}\ 1\leq k\leq 2$ and $q(e_{1}, e_{2})=0$.  Therefore  $(V, q)$,
is  isometric to $(R^{2}, x^{2}+y^{2})$.

It is easy to check that the $R$-linear map defined by
$$\alpha(\varepsilon _{1})=\varepsilon _{1}\otimes t, \ \ \alpha (\varepsilon _{2})=\varepsilon _{2}\otimes t^{n-1}$$
induces an $A$-comodule structure on $V$,  and therefore an $A^{D}$-module structure described by

$$ u\varepsilon _{1}=u(t)\varepsilon _{1}, \mbox \ \ \   u\varepsilon _{2}=u(t^{n-1})\varepsilon _{2},
 $$
 for all $u \in A^{D}$.
It is  then verified that indeed
$$q(u\varepsilon _{k}, \varepsilon _{k})=q(\varepsilon _{k}, S^{D}(u)\varepsilon _{k})=0,\ \mbox{for}\ 1\leq k \leq 2$$
and $$q(u\varepsilon _{1}, \varepsilon _{2})=q(u(t)\varepsilon _{1}, \varepsilon _{2})=
q(\varepsilon _{1}, u(t)\varepsilon _{2})=q(\varepsilon _{1}, S^{D}(u)(t^{n-1})\varepsilon _{2})=
q(\varepsilon _{1}, S^{D}(u)\varepsilon _{2})\ \ , $$
for all $u \in A^{D}$. We deduce from the previous relations that $(V, q)$ is an $A$-equivariant symmetric bundle.
We may associate to this form a morphism of group schemes,
 $\rho: \mu_{n}\rightarrow {\bf O}(q) $,  defined for any $R$-algebra $C$ by the
group homomorphism $\rho_{C}: \mu _{n}(C)\rightarrow O(q_{C})$ where
$\rho_{C}(\xi )$ is given  on the basis $\{\varepsilon _{1}, \varepsilon _{2}\}$ by
$$\rho_{C}(\xi )(\varepsilon _{1})=
\xi \varepsilon _{1}, \ \ \  \rho_{C}(\xi )(\varepsilon _{2})=\xi^{-1}\varepsilon _{2}\ . $$
When $i \in R$, this is a representation of $\mu_{n}$ into ${\bf O_{2}}$.

Our goal is now to study the twists of this  form by torsors.  We now assume that   $R$ is  a discrete valuation ring of characteristic $0$
and of residual characteristic $p$ different from $2$.  We take  $A=((R[T]/(T^{p}-1))= R[t])$.
 For any unit $y\in R^{\times}$,   we let $B_{y}$ be the $R$-algebra $R[X]/(X^{p}-y)=R[x]$.
The $R$-linear map $\alpha: B_{y}\rightarrow B_{y}\otimes A$ defined by $\alpha(x^{k})=x^{k}\otimes t^{k},   0\leq k\leq p-1, $
endows
$B_{y}$ with a structure of a $A$-comodule algebra and so with a structure of  $A^{D}$-module given by

$$ub=\sum_{0\leq k\leq p-1}u(t^{k})b_{k}x^k, \ u \in A^{D}, b=\sum b_{k}x^{k}\ . $$
 Let  $z$ be a $p$-th root of $y$ in an algebraic closure of the fraction field of $R$ and let $C$ be the algebra $R[z]$.   The map
$T\rightarrow zX$ induces an isomorphism of $C$-$A$ comodule algebras  from $C\otimes A$ onto $C\otimes B_{y}$.
This proves that $B_{y}$ is a PHS for $A$. By hand checking we verify that

$$ I(A)=R(1+ ...+ x^{p-1})=pR((1+ ...+ x^{p-1})/p), $$
where $(1+ ...+ x^{p-1})/p$ is an idempotent of $A_{K}$. We then get that
$$\Lambda=\varepsilon (I(A))=pR\ . $$
We now assume that {\it $R$ contains a square root of $p$}, denoted by $p^{1/2}$. Since $A$ now satisfies ${\bf H_{2}}$ for any
unit $y\in R^{\times}$ we may consider  the projective symmetric bundle $(\mathcal{D}^{-1/2}(B_{y}), Tr)=(p^{-1/2}B_{y}, Tr)$.
We note  that $Tr(x^{k}, x^{l})=1$ (resp. $py$) for $k=l=0$
(resp. $k+l=p$)  and $0$ otherwise. It follows from Definition 3 of Section 3 that the twist of $(V, q)$ by $B_{y}$ is the
$R$-symmetric bundle

$$(V_{y}, q_{y})= ((p^{-1/2}R[x]\otimes V))^{A}, (Tr\otimes q)^{A})\  . $$

\begin {prop} For any unit $y\in R$ there exists an isometry of $R$-symmetric bundles:
$$(V, q)\simeq (V_{y}, q_{y}) \ .$$
\end{prop}
\begin{proof}  The set
 $$\{p^{-1/2}x^{k}\otimes \varepsilon _{1}, p^{-1/2}x^{k}\otimes \varepsilon _{2},  1\leq k \leq p-1 \}$$ is an $R$-basis of
 $(p^{-1/2}R[x]\otimes V))$.
 For any $k, 0\leq k \leq p-1$, denote  by $u_{k}$
the element of $A^{D}$ defined by  $u_{k}(t^{l})=\delta _{kl}$. It is immediate that $\{u_{0}, ..., u_{p-1}\}$ is an $R$-basis of
$A^{D}$ while  $u_{0}$ is  a basis for the module of integral of $A^{D}$. It follows from Lemma 2.5   that
$$\{u_{0}(p^{-1/2}x^{k}\otimes \varepsilon _{1}), u_{0}(p^{-1/2}x^{k}\otimes \varepsilon _{2}),  1\leq k \leq p-1 \}$$ is
a set of generators of  $V_{y}$. Since
$\Delta ^{D}(u_{0})= u_{0}\otimes u_{0}+ u_{1}\otimes u_{p-1}+ ...+ u_{p-1}\otimes u_{1}$ and  since $A^{D}$ acts
diagonally we obtain that
$$u_{0}(p^{-1/2}x^{k}\otimes \varepsilon _{j})= (u_{0}p^{-1/2}x^{k}\otimes u_{0}\varepsilon _{j}) +
\sum_{1 \leq l \leq p-1}(u_{l}p^{-1/2}x^{k}\otimes u_{p-l}\varepsilon _{j}), 1 \leq j \leq 2\ .$$
For $j=1$ we deduce that $u_{0}(p^{-1/2}x^{k}\otimes \varepsilon _{1})=p^{-1/2}x^{p-1}\otimes \varepsilon _{1}$
 if $k= p-1$ and $0$ otherwise. For $j=2$ then
 $u_{0}(p^{-1/2}x^{k}\otimes \varepsilon _{2})=p^{-1/2}x\otimes \varepsilon _{2}$ if $k=1$ and $0$ otherwise.
  This implies that
 $\{\varepsilon '_{1,y}=p^{-1/2}x^{p-1}\otimes \varepsilon _{1}, \varepsilon '_{2,y}=p^{-1/2}x\otimes \varepsilon _{2}\}$
 is a $R$-basis of $V_{y}$
such that
 $$u_{0}\varepsilon '_{1,y}=\varepsilon '_{1,y}, \ u_{0}\varepsilon '_{2,y}=\varepsilon '_{2,y}\ .$$
It follows from the definition of $q$ and $q_{y}$  that
 $$q_{y}(u_{0}\varepsilon '_{j,y}, u_{0}\varepsilon '_{j,y})=(Tr\otimes q)(\varepsilon '_{j,y}, u_{0}\varepsilon '_{j,y})
 =(Tr\otimes q)(\varepsilon '_{j,y}, \varepsilon '_{j,y})= 0, 1\leq j\leq 2\  $$
 and
 $$q_{y}(u_{0}\varepsilon '_{1,y}, u_{0}\varepsilon '_{2,y})=(Tr\otimes q)(\varepsilon '_{1,y}, u_{0}\varepsilon '_{2,y})
 =(Tr\otimes q)(\varepsilon '_{1,y}, \varepsilon '_{2,y})=y/2\ .$$
 We conclude  that the $R$-linear map given  by $(\varepsilon_{1}\mapsto  \varepsilon'_{1,y}) $ and
 $(\varepsilon_{2}\mapsto  y^{-1}\varepsilon'_{2,y}) $ induces  an isometry from $(V, q)$ onto $(V_{y}, q_{y})$
 for any $y\in R^{*}$.
 \end{proof}

 \subsection {A dihedral representation}

 We construct an example of an orthogonal representation of a non-commutative group scheme which induces,  by restriction to the
 generic fiber,  a dihedral representation as defined in [F].

 \subsubsection {The Hopf algebra}

Let $\ D$ denote a dihedral group of order $2n$ with $n>2$ and with
generators and relations:%
$$D=\left\langle \sigma ,\tau \mid \ \sigma ^{n}=1=\tau ^{2},\ \tau \sigma
\tau =\sigma ^{-1}\right\rangle .$$
$R$ denotes an integral domain which is regular local ring and which is a
finitely generated algebra over $\mathbf{Z}_{p}$ for an odd prime $p.$ \ We
let $K$ denote the field of fractions of $R$ and we suppose that $\mu
_{n}\subset R.\ $We let $H_{K}$ denote the group algebra$\ K\left[ D\right]
, $ endowed with its structure of  non-commutative Hopf algebra.  For any character $\phi $ of $\left\langle \sigma \right\rangle
$ we let
$e_{\phi}$ be the idempotent $n^{-1}\sum_{\varsigma \in <\sigma>}\phi(\varsigma)\varsigma^{-1}$ of $K\left\langle \sigma \right\rangle
$.
We consider    the split maximal $\ R$-order $\mathcal{M}$ in the
group algebra $K\left\langle \sigma \right\rangle $%
$$
\mathcal{M}=\oplus _{\phi }R.e_{\phi }\supset R\left\langle \sigma \right\rangle
$$
where $\phi $ ranges over the abelian\ $K$-valued characters of the cyclic
group $\left\langle \sigma \right\rangle .$ Recall that $\mathcal{M}$ is an $R$-Hopf
order in the group algebra $K\left\langle \sigma \right\rangle $ with
$$
\Delta (e_{\phi }) =\sum_{\alpha, \beta\mid \alpha.\beta=\phi} e_{\alpha}\otimes e_{\beta}.
$$
We then let $H$ denote the$\ R$-order in $K\left[ D\right] \ $given by the
twisted group ring $\mathcal{M}\circ \left\langle \tau \right\rangle ;$ so that we may
write%

$$\mathcal{M}\circ\left\langle \tau \right\rangle
 =
R\left\langle \tau \right\rangle \oplus _{\phi }^{\prime }R\circ _{\phi
}\left\langle \tau \right\rangle  \ \mbox{if }n\mbox{ is odd}$$
and
$$\mathcal{M}\circ \left\langle \tau \right\rangle=R\left\langle \tau \right\rangle
\oplus R.e_{\theta }\left\langle \tau \right\rangle \oplus _{\phi }^{\prime
}R\circ _{\phi
}\left\langle \tau \right\rangle   \ \mbox{if }n\mbox{ is even, }$$
where$\ \theta $ is the unique quadratic character of $\left\langle \sigma
\right\rangle $ if $n\ $is even,  where $\oplus _{\phi }^{\prime }$
denotes the  sum over the orbits of abelian
characters $\phi \ $of order greater than 2, modulo the action of the
involution $\sigma \longmapsto \sigma ^{-1} $ and where we have set
$$R\circ _{\phi
}\left\langle \tau \right\rangle=(Re_{\phi}+Re_{\overline \phi})\circ\left\langle \tau \right\rangle\ \mbox{with}\ \tau e_{\phi}=e_{\overline \phi}\tau\ .$$
For $\phi $ with order greater than 2 recall from [R], Chapter 9, that we know that\ $%
R\circ _{\phi }\left\langle \tau \right\rangle $ is a hereditary $R$-order,
and is maximal if $n$ is invertible in $R$.

\begin{lem}
$H$ is an $\ R$-Hopf order in $\ H_{K}.$
\end{lem}

\begin{proof} Basically we need to show that $\ \Delta \left( H\right) \subset
H\otimes H.$ Since $H$ is generated over $\ R$ by $\tau $ and the various $%
e_{\phi },$ it will suffice to show that
$$
\Delta (\tau) \in H\otimes H,\ \mbox{and }\ \Delta (
e_{\phi }) \in H\otimes H.
$$
The first follows from the definition of $\Delta $ and the second  from our previous equalities.
\end{proof}

\bigskip

Henceforth we identify $\ H_{K}^{D}=\mathrm{Map}\left( D,K\right) ;$ so that
Spec$\left( H_{K}^{D}\right) $ is the constant group scheme over $\mathrm{%
Spec}\left( K\right) $ associated to $D.$ We then define$\ A$ \ to be the $\
R$-dual \ $\ $%
$$
A=H^{D}=\mathrm{Hom}_{R}\left( H, R\right) ;
$$
then$\ \mathrm{Spec}\left( A\right) $ is a non-constant (but generically
constant) group scheme over $\mathrm{Spec}\left( R\right) $. We note that, since $H$ is a finitely generated projective $R$-module, it follows that
$A^{D}=(H^{D})^D$ is naturally isomorphic to $H$. We will identify these Hopf algebras.

\begin{lem}
Since $\ 2$ is invertible in $R$, the modules of integrals for the Hopf
orders $\ H$ and $A$ are:
\begin{itemize}
\item[i)]  $I(H) =2R.e_{D}=n^{-1}R.\sum_{d\in D}d ; $

\item[ii)]  $I(A)=Rn.l_{0}$ where for $d\in D$, $l_{0}(d)=1$ if  $d=1_{D}$ and $0$ otherwise.
\end{itemize}
Moreover,  for any principal homogeneous space $B$ for $A$, we have the equality:
$$\mathcal{D}_{B}=nB$$

\end{lem}
\begin{proof} It follows from Corollary 3.5 that $i)$ and $ii)$ are equivalent. We shall prove $i)$. Let $\phi_{0}$ be the trivial character
of $<\sigma>$. We  observe that
$\varepsilon (e_{\varphi_{0}})=1$ while $\varepsilon (e_{\phi})=0\ \forall \phi\not\neq \phi_{0}$. Therefore
 $x\in I(H)$ if and only if
$$e_{\phi}x=0\ \forall \phi\not\neq \phi_{0}, \ e_{\phi_{0}}x=x\ \mbox{and}\  \tau x=x \  .$$ We deduce
immediately from  these equalities that $x\in 2Re_{D}$ as required. Let $B$ be a PHS for $A$. It follows from  $ii)$ and Corollary 3.3
 that
$\mathcal{D}_{B}=nB$.
\end{proof}
\vskip 0.1 truecm
\noindent {\bf Remark.}\ Note that instead of working with the hereditary order $H$
we could just as well work with any Hopf order $\ H^{\prime }$ in $\ H_{K\ }$%
 containing $R\left[ D\right] ;$
then $\mbox{Spec}\left( A'=H^{\prime D}\right) $ will be a non-constant (but
generically constant) group scheme.  The case of a suitable maximal $%
R $-order in $H_{K}\ $ could also be particularly interesting.
\subsubsection{The equivariant symmetric bundle.}

We fix an
abelian character $\chi $ of the cyclic group $\left\langle \sigma
\right\rangle $ with order greater than 2, and we consider the quadratic $R$-module$\ \left( M, q\right) $:

$$
\ M=R\left\langle \sigma \right\rangle \circ \left\langle \tau \right\rangle
.e_{\chi }=R.e_{\chi }+R.\tau e_{\chi }=R.e_{\chi }+R.e_{\overline{\chi }%
}\tau $$
endowed with the quadratic form%
$$
q\left( x,y\right) =\frac{1}{2}tr\left( x.\tau .\overline{y}\right),
$$
where $tr:K\left[ D\right] \rightarrow K$ denotes the usual trace map where for $d\in D$

$$ tr\left( d\right) =2n\ \mbox{\ if }d=1_{D} \ \mbox{and}\ 0\ \mbox{otherwise}\ .$$

\begin {lem} $(M, q)$ is an $A$-equivariant symmetric bundle.
\end{lem}

\begin{proof} It is immediate from the definition that $M$ is an $H=A^{D}$-module.
 Note that
$$
q\left( dx,dy\right) =\frac{1}{2}tr\left( dx.\tau .\overline{y}d^{-1}\right)
=\frac{1}{2}tr\left( x.\tau .\overline{y}\right) =q\left( x,y\right)
$$
so that $q$ is indeed $D$-invariant. Of course we also have
$$
q\left( e_{\chi },e_{\overline{\chi }}.\tau \right) =\frac{1}{2}tr\left(
e_{\chi }\tau ^{2}e_{\chi }\right) =\frac{1}{2}tr\left( e_{\chi }\right) =1$$
$$
q\left( e_{\chi },\ e_{\chi }\right) =0=q\left( e_{\overline{\chi }}.\tau
,\ e_{\overline{\chi }}.\tau \right) ;
$$
and so $q$ is an$\ R$-perfect pairing on $M$, and in fact is seen to have discriminant -1.
\end{proof}

\subsubsection{Twists of the form}

Let $B$ be a PHS for $A$ over $R$. The structure map
$$\alpha_{B}: B\rightarrow B\otimes A$$
induces an isomorphism of $B$-algebras and $H$-modules
$$i_{d}\otimes \alpha_{B}: B\otimes_{R} B\simeq B\otimes_{R} A$$
(recall that H acts on each side via the right-hand factors).
We put $C=\left\langle \sigma \right\rangle$ and set $E= B^{C}$.

\begin{prop}
$E $ is a PHS for $A^{C}$  and $\mathcal{D}_{A^{C}/R}=A^{C}.$
\end{prop}
\begin{proof} Because $B$ and $A$ are $R$-flat we have the
isomorphism induced by taking the $C$-fixed points:
$$\beta: B\otimes_{R}E\simeq B\otimes_{R}A^{C} .$$
We know that $B$ is finite and flat and hence faithfully flat over $R$. Moreover it follows from the definitions that $A^{C}$
is a finite and free $R$-module. This implies that $B\otimes_{R}A^{C}$ and thus $B\otimes_{R}E$ is flat over $B$ and
so that $E$ is flat over $R$ ([W], XIII, 1.3.3). We have therefore shown that $E$ is a commutative, finite and flat $R$-algebra.

Let $q:D\rightarrow \left\langle \tau \right\rangle $ be the quotient group
homomorphism with kernel $C.$ Because 2 is invertible in $R$, we know that $%
R\left\langle \tau \right\rangle $ is the unique maximal $R$-order in $%
K\left\langle \tau \right\rangle;$ and so in particular we see that $%
\mathrm{Spec}\left( R\left\langle \tau \right\rangle \right) $ is a closed
subgroup scheme of $\mathrm{Spec}(H)) .$ We recall the inclusions
$$A^{C}\subset A=\mbox {Hom}_{R}(H, R)\subset A_{K}=\mbox {Hom}_{K}(K[D], K) .$$
The group $D$ acts on $A_{K}$ via the rule that for all $f \in A_{K}$, and for all $\gamma\in D$
$$f^{\gamma}: \alpha\mapsto f(\alpha\gamma^{-1}).$$
It therefore follows that for any $f\in A$ and for any character $\phi$ of $C$ we have:
$$f^{\sigma}(e_{\phi})=f(e_{\phi^{\sigma^{-1}}})=\phi(\sigma^{-1})f(e_{\phi})\ \mbox{and}\ f^{\sigma}(e_{\phi}\tau)=
\phi(\sigma)f(e_{\phi}\tau)\ .$$
Therefore,  by using the description of $H$ above, we deduce that  dually
$\mathrm{Map}\left( \left\langle \tau \right\rangle ,R\right) \ $%
identifies as $A^{C}$ and so $\mathrm{Spec}\left( A^{C}\right) $
identifies as a quotient group scheme of $\mathrm{Spec}\left( A\right) $.

We now observe that $\beta$ induces an action map
$$\gamma: E\otimes_{R}E\rightarrow E\otimes_{R}A^{C} .$$
In order to show that $\gamma$ is an isomorphism it suffices to prove that $B$ is faithfully flat over $E$. One
checks easily that $A$ is free over $A^{C}$. Therefore $B\otimes_{R}A$ is free over $B\otimes_{R}A^{C}$ and similarly $B\otimes_{R}B$ is free over $B\otimes_{R}E$. Using once again that $B$ is faithfully flat over $R$ we deduce that $B$ is
flat over $E$. Since $B$ is finite over $E$ we conclude that it is faithfully flat over $E$ and thus that $\gamma$ is an isomorphism.
We have proved that $E$ is a PHS for $A^{C}$. Since by definition $A^{C}$ is \'etale over $R$, then $\mathcal{D}_{A^{C}/R}=A^{C}$
\end{proof}

We recall that  $\mathcal{M}$ denotes the split maximal $R$-order in $K[C]$.  If we let $L$ denote the ring of fractions of $E$ then  $%
\mathcal{M}_{E}=E\otimes _{R}\mathcal{M}$ is the split maximal order in $L[C]$.  We consider the duals
$\ \mathcal{N}=\mathrm{Hom}_{R}\left( \mathcal{M},R\right) $ and
$\mathcal{N}_{E}=\mathcal{N}\otimes _{R}E=\mathrm{%
Hom}_{E}\left( \mathcal{M}, E\right) ;$ by duality these are the minimal Hopf orders in
$\mbox{Map}\left( C,K\right) $ and  $\mbox{Map}\left( C,L\right) $ respectively. Then we have
$$\mathcal{D}_{\mathcal{N}/R}=n\mathcal{N}\ \mbox{and}\ \mathcal{D}_{\mathcal{N}_{E}/E}=n\mathcal{N}_{E}\ .$$

\begin {prop} $B$ is a PHS of $\mathcal{N}_{E}$ over $E$.
\end{prop}
\begin{proof}
The inclusion map $C\hookrightarrow D$ induces dual maps

$$\begin{array}{ccc}
\mathcal{M} & \hookrightarrow & H \\
\downarrow &  & \downarrow \\
KC & \hookrightarrow & KD%
\end{array} $$
$$
\begin{array}{ccc}
A & \rightarrow & \mathcal{N} \\
\downarrow &  & \downarrow \\
\mathrm{Map}\left( D,K\right) & \rightarrow & \mathrm{Map}\left( C,K\right)%
\end{array} $$
and the isomorphism $i_{d}\otimes \alpha_{B}$ induces a map
$$
B\otimes _{R}B\rightarrow B\otimes _{R}A\rightarrow B\otimes _{R}\mathcal{N}\cong
B\otimes _{E}E\otimes _{R}\mathcal{N}\cong B\otimes _{E}\mathcal{N}_{E};
$$
as $\ \mathcal{N}$ acts trivially on $E$ this map actually factors through $B\otimes
_{E}B,$ and so in summary we have produced the action map
$$
B\otimes _{E}B\rightarrow B\otimes _{E}\mathcal{N}_{E} \ .$$
In order to show that this injective map is in fact an isomorphism we shall
show that their discriminants coincide. By the tower formula we know that
$$
\mathcal{D}_{B/E}=\mathcal{D}_{B/R}\mathcal{D}_{E/R}^{-1}=\mathcal{D}%
_{B/R}=nB;
$$
here the second equality comes from Proposition 5.6 and the third equality
comes from Lemma 5.4. Since we know that $\mathcal{D}_{\mathcal{N}_{E}/E}=n\mathcal{N}_{E}$ we  conclude  that
the map  is
indeed an isomorphism.
\end{proof}

Next we consider\ the ring $E$ over the maximal order $R\left\langle \tau
\right\rangle .\ $By the above we may write $\ E$ as a direct sum \ of two
rank one free $\ R$-modules $\ E=E_{+}\oplus E_{-}$ where $\tau $ acts on $\
E_{+}=R$ trivially and on $E_{-}$ by $-1.$ Choosing a generator $\delta $
for $E_{-}$ over $\ R,$ we have an element of$\ E$ with the property that $%
\delta ^{2}\in R;$ moreover, as $\ E$ \ is a $\left\langle \tau
\right\rangle $-torsor, we know that in fact $\ \delta ^{2}\in R^{\times }.$

We now apply a somewhat similar analysis to $\ B\ $ viewed initially as an $%
\ \mathcal{M}_{E}$-module. We may then write $\ B=\oplus B_{\chi }$ with $\chi $
ranging over the abelian characters of $\left\langle \sigma \right\rangle ,$
and with each $B_{\chi }$ a free rank one $R$-module, with generator $%
t_{\chi }$ and with $t_{\chi }t_{\phi }$ divisible by $t_{\chi \phi }\ $(see
2.e in [CEPT])$;$ since$\ B/E\ $ is an $\ \mathcal{N}_{E}$-torsor by Proposition 5.7 we
can write:

$$
B=E[X] \mbox{mod}( X^{n}-\alpha ^{n})$$
for some $\alpha ^{n}\in E^{\times }, $ with  $
a=\alpha .^{\tau }\alpha \in R^{\times }.\ $

We now assume that $A$ satisfies ${\bf H_{2}}$,  (Definition 2, Section 3.2), which reduces,  in this particular case, to requiring  that $nR$  is
 the square of a principal  ideal.  For the sake of simplicity we shall assume that $n$ is  a square of $R$ when it is not a unit. We denote by
 $n^{1/2}$ a square root of $n$. Moreover we choose $\chi$ as  $\left\langle \sigma \right\rangle$-character of $\alpha$. These preparations being in place, we can now determine the twist of $(M, q)$ by $B$
 which is defined, according to Definition 3,  by:

$$ (\tilde M_{B}, \tilde q_{B})= (cB\otimes_{R}M, Tr\otimes q)^{A}, $$
where $c=1$ (resp. $n^{-1/2}$) if $n$ is a unit (resp. otherwise).

\begin{prop}We  have the following equalities:
\begin{itemize}
\item[i)]
$\tilde M_{B}=R\varepsilon_{1}\oplus R\varepsilon_{2}$, where we set
$$\varepsilon_{1}=c\alpha^{\tau}\otimes e_{\chi}+c\alpha\otimes e_{\bar \chi}\tau
\ \mbox{and}\ \varepsilon_{2}=c\delta\alpha^{\tau}\otimes e_{\chi}-c\delta\otimes e_{\bar\chi}\tau .$$

\item[ii)] $$\tilde q_{B}(\varepsilon_{1}, \varepsilon_{1})=2a, \ \tilde q_{B}(\varepsilon_{2}, \varepsilon_{2})=
-2a\delta^{2}, \ \tilde q_{B}(\varepsilon_{1}, \varepsilon_{2})=0 .$$
\end{itemize}
\end{prop}
\begin{proof} We recall from the definition that
$$\tilde M_{B}= (cB\otimes_{R}M)^{A}=\{z\in cB\otimes_{R}M\ \mid uz=\varepsilon ^{D}(u)z\ \forall u\in A^{D}\} .$$
One easily checks  from the definition of $A^{D}=H$ that
$$(cB\otimes_{R}M)^{A}=(cB\otimes_{R}M)^{D} . $$ We now observe that Propositions 5.6 and 5.7 provide us with a free $R$-basis
of $B$. Moreover,  we know the action of $D$ on the elements of this basis. Hence by a straightforward computation
we obtain the equality:
$$(cB\otimes_{R}M)^{C}=R(c\alpha^{\tau}\otimes e_{\chi})+R(c\delta\alpha^{\tau}\otimes e_{\chi})
+R(c\alpha\otimes e_{\bar \chi}\tau) +R(c\delta \alpha\otimes e_{\bar \chi}\tau) .$$ It now suffices to take the $c$-fixed
points of the right-hand side of the equality above to obtain $i)$. In order to prove $ii)$ we shall assume that $n$ is not a unit; the easier case
where $n$ is a unit is left to the reader. From the definitions we obtain that
$$(Tr\otimes q)(\varepsilon_{1}, \varepsilon_{1})=2Tr(n^{-1}\alpha\alpha^{\tau})q(e_{\chi}, e_{\bar \chi}\tau)=4a,\
(Tr\otimes q)(\varepsilon_{1}, \varepsilon_{2})=0 $$ and
$$(Tr\otimes q)(\varepsilon_{2}, \varepsilon_{2})=-2Tr(n^{-1}\delta^{2}\alpha\alpha^{\tau})q(e_{\chi}, e_{\bar \chi}\tau)
=-4a\delta^{2} . $$
Finally we have to compare the forms $Tr\otimes  q$ and $(Tr\otimes q)^{A}$ on $\tilde M_{B}$. Using Lemma 2.5 and Lemma 5.4, we note that
$\tilde M_{B}=\theta^{D}M_{B}$, with $\theta^{D}=n^{-1}\sum_{u\in D}u$. Then,  for any element $m$ and $n$ in
$M_{B}$, we have:
$$(Tr\otimes  q)(\theta^{D}m, \theta^{D}n)=n^{-1}\sum_{u\in D}(Tr\otimes  q)(um, \theta^{D}n)$$
$$=2(Tr\otimes  q)(m, \theta^{D}n)=2(Tr\otimes  q)^{A}(\theta^{D}m, \theta^{D}n) .$$
We conclude that $\tilde q_{B}=\frac{1}{2}(Tr\otimes  q)$ and so $ii)$ follows from the previous equalities.
\end{proof}
\medskip

\noindent {\bf Remark} We observe that the discriminant of the form $\tilde q_{B}$ is equal to $-\delta^{2}$,  up to a square.
Therefore, if $-1$ is a square of $R$, since we know that $\delta^{2}$ is not a square of $R$,  we deduce that the discriminant of $\tilde q_{B}$
is not a square and thus that the forms $q$ and $\tilde q_{B}$ are not isometric.

\bigskip

philippe.cassou-nogues@math.u-bordeaux1.fr

\smallskip

ted@math.upenn.edu

\smallskip

baptiste.morin@math.u-bordeaux1.fr

\smallskip

martin.taylor@merton.ox.ac.uk

\end{document}